\documentclass[a4paper,11pt,reqno]{article}

\usepackage[utf8]{inputenc}
\usepackage[T1]{fontenc}
\usepackage{lmodern}
\usepackage[english]{babel}
\usepackage{microtype}

\usepackage{amsmath,amssymb,amsfonts,amsthm}
\usepackage{mathtools,accents}
\usepackage{mathrsfs}
\usepackage{aliascnt}
\usepackage{braket}
\usepackage{bm}
\usepackage{esint}

\usepackage[a4paper,margin=3cm]{geometry}
\usepackage[citecolor=blue,colorlinks]{hyperref}

\usepackage{enumerate}
\usepackage{xcolor}
\makeatletter
\g@addto@macro\@floatboxreset\centering
\makeatother

\makeatletter
\def\newaliasedtheorem#1[#2]#3{
  \newaliascnt{#1@alt}{#2}
  \newtheorem{#1}[#1@alt]{#3}
  \expandafter\newcommand\csname #1@altname\endcsname{#3}
}
\makeatother

\numberwithin{equation}{section}

\newtheoremstyle{slanted}{\topsep}{\topsep}{\slshape}{}{\bfseries}{.}{.5em}{}

\theoremstyle{plain}
\newtheorem{theorem}{Theorem}[section]
\newaliasedtheorem{proposition}[theorem]{Proposition}
\newaliasedtheorem{lemma}[theorem]{Lemma}
\newaliasedtheorem{corollary}[theorem]{Corollary}
\newaliasedtheorem{counterexample}[theorem]{Counterexample}

\theoremstyle{definition}
\newaliasedtheorem{definition}[theorem]{Definition}
\newaliasedtheorem{question}[theorem]{Question}
\newaliasedtheorem{example}[theorem]{Example}
\newaliasedtheorem{conjecture}[theorem]{Conjecture}

\theoremstyle{remark}
\newaliasedtheorem{remark}[theorem]{Remark}
\newaliasedtheorem{comment}[theorem]{Comment}

\newcommand{\setN}{\mathbb{N}}

\newcommand{\setR}{\mathbb{R}}

\let\phi\varphi

\newcommand{\di}{\mathop{}\!\mathrm{d}}

\DeclareMathOperator{\supp}{supp}

\newcommand{\Ch}{{\sf Ch}}
\DeclareMathOperator{\Lip}{Lip}

\newcommand{\dist}{\mathsf{d}}

\newcommand{\meas}{\mathfrak{m}}

\DeclareMathOperator{\RCD}{RCD}

\DeclareMathOperator{\vol}{\mathrm{vol}}

\newfont{\tmpf}{cmsy10 scaled 2500}

\begin{document}
\title{A characterization of non-collapsed $\RCD(K,N)$ spaces via Einstein tensors}
\author{
Shouhei Honda
\thanks{Tohoku University, \url{shouhei.honda.e4@tohoku.ac.jp}}
  \and
Xingyu Zhu
\thanks{Georgia Institute of Technology, \url{xyzhu@gatech.edu}} } \maketitle

\begin{abstract}
 We investigate the second principal term in the expansion of metrics $c(n)t^{(n+2)/2}g_t$ induced by heat kernel embedding into $L^2$ on a compact $\RCD(K,N)$ space. We prove that the divergence free property of this term holds in the weak, asymptotic sense if and only if the space is non-collapsed up to multiplying a constant to the reference measure. This seems new even for weighted Riemannian manifolds.  {Moreover an example tells us that the result cannot be generalized to the noncompact case. In this sense, our result is sharp.}
\end{abstract}

\tableofcontents
\section{Introduction}

For a closed Riemannian manifold $(M^n ,g)$, the \textit{Einstein tensor} $G^g$ is defined by
\begin{equation}
  G^g:=\mathrm{Ric}^g-\frac{1}{2}\mathrm{Scal}^gg,
\end{equation}
where $\mathrm{Ric}^g$ and $\mathrm{Scal}^g$ denote the Ricci and the scalar curvature, respectively. 
It is well-known that $G^g$ is divergence free:
\begin{equation}\label{eq:divzero}
\nabla^*G^g=0
\end{equation}
which is a direct consequence of the Bianchi identity.
\par The main purpose of the paper is to establish (\ref{eq:divzero}) for so-called \textit{non-collapsed $\RCD(K, N)$ spaces.} More precisely, for a compact $\RCD(K, N)$ space $(X, \dist, \meas)$, (\ref{eq:divzero}) holds in some sense as explained below if and only if $(X, \dist, \meas)$ is non-collapsed up to multiplication of a positive constant to the measure $\meas$. It is worth pointing out that our argument allows us to provide a new proof of (\ref{eq:divzero}) even for a closed Riemannian manifold $(M^n ,g)$ without using the Bianchi identity.

In order to explain how to justify (\ref{eq:divzero}), let us recall B\'erard-Besson-Gallot's work in \cite{BerardBessonGallot}.
They proved that for a closed Riemannian manifold $(M^n, g)$ and fixed $t \in (0, \infty)$, the map $\Phi_t$ from $M^n$ to $L^2(M^n, \mathrm{vol}_g)$ defined by
     \begin{equation}
        x \mapsto (y \mapsto p(x, y, t))
     \end{equation}
is a smooth embedding with the following asymptotic expansion:
  \begin{equation}\label{eq:bbg}
      c(n)t^{(n+2)/2} \Phi_t^*g_{L^2} =  g - \frac{2t}{3}G^g + O(t^2)
  \end{equation}
as $t \to 0^+$, where $p(x, y, t)$ denotes the heat kernel of $(M^n ,g)$ and $c(n)$ is a positive constant depending only on $n$ defined by
  \begin{equation}
     c(n):=(4\pi)^n \left(\int_{\mathbb{R}^n}|\partial_{x_1}(e^{-|x|^2/4})|^2\di x\right)^{-1}=4(8\pi)^{n/2}.
  \end{equation}
Let us denote $g_t=\Phi_t^*g_{L^2}$ and let us remark that %$g_t$ can be written by
    \begin{equation}
       g_t=\int_{M^n}\di_xp \otimes \di_xp \di \mathrm{vol}^g(y).
    \end{equation}
By (\ref{eq:bbg}) we see that as $t \to 0^+$

  \begin{equation}\label{eq:smoothratio}
     \frac{c(n)t^{(n+2)/2}g_t-g}{t} \to -\frac{2}{3}G^g.
  \end{equation}
Since the convergence of (\ref{eq:smoothratio}) is uniform on $M^n$ by the proof (see for instance Theorem \ref{thm:bbgweighted}), (\ref{eq:divzero}) can be reformulated by
  \begin{equation}\label{eq:asym}
    \int_{M^n}\left\langle\frac{c(n)t^{(n+2)/2}g_t-g}{t}, \nabla \omega \right\rangle \di \mathrm{vol}^g \to 0 
  \end{equation}
as $t \to 0^+$ for any smooth $1$-form $\omega$ on $M^n$, where $\langle \cdot, \cdot\rangle$ denotes the canonical inner product on $T_x^*M^n\otimes T_x^*M^n$ for each $x \in M^n$. In this paper the sequence on the (LHS) of (\ref{eq:smoothratio}) is called \textit{weakly asymptotically divergence free}, if it satisfies (\ref{eq:asym}) for any smooth $1$-form $\omega$ on $M^n$, see Definition \ref{def:asy}.

Next let us introduce a recent work of Ambrosio-Portegies-Tewodrose and the first author \cite{AHPT}, which partially generalizes Bérard-Besson-Gallot's result (\ref{eq:bbg}) to $\RCD(K, N)$ spaces which give a special class of metric measure space having lower bounds on Ricci curvature in a synthetic sense introduced in \cite{AmbrosioGigliSavare14} by Ambrosio-Gigli-Savar\'e when $N=\infty$, in \cite{Gigli13, Gigli1} with introducing the infinitesimal Hilbertian condition by Gigli, in \cite{ErbarKuwadaSturm} by Erbar-Kuwada-Sturm, in \cite{AmbrosioMondinoSavare} by Ambrosio-Mondino-Savar\'e, when $N<\infty$. 

Roughly speaking a metric measure space is said to be an $\RCD(K, N)$ space if  the $H^{1, 2}$-Sobolev space is a Hilbert space and the following holds;
\begin{itemize}
\item  the Ricci curvature is bounded below by $K$, and the dimension is bounded above by $N$, in a synthetic sense via optimal transportation theory by Lott-Sturm-Villani \cite{LottVillani, Sturm06a, Sturm06b}. %so-called ``\textit{curvature-dimension condition} $\CD(K, N)$'' introduced in \cite{LottVillani} by Lott-Villani and \cite{Sturm06a, Sturm06b} by Sturm independently
\end{itemize}
Typical examples include measured Gromov-Hausdorff limit spaces of Riemannian manifolds with uniform lower bounds on Ricci curvature, so-called Ricci limit spaces, and weighted Riemannian manifolds $(M^n, \dist^g, \mathrm{vol}^g_f)$, where $f \in C^{\infty}(M^n)$ and $\mathrm{vol}^g_f=e^{-f}\mathrm{vol}^g$. 

Thanks to recent quick developments on the study of $\RCD(K, N)$ spaces,
%we knew
 many structure results on such spaces are known. For example, it is proved in \cite{BrueSemola} by Bru\`e-Semola that for any $\RCD(K, N)$ space, where $N<\infty$, there exists a unique integer $n$, so-called the essential dimension, such that for almost every point of the space, the tangent cone at the point is unique and is isometric to the $n$-dimensional Euclidean space.

On the other hand, a restricted class of $\RCD(K, N)$ spaces, so-called ``\textit{non-collapsed}'' $\RCD(K,N)$ spaces, is introduced in \cite{DePhillippisGigli} by DePhilippis-Gigli as a synthetic counterpart of non-collapsed Ricci limit spaces. The definition is that the reference measure coincides with the $N$-dimensional Hausdorff measure. Then it is known that non-collapsed $\RCD(K,N)$ spaces have finer properties
%rather
 than that of general $\RCD(K, N)$ spaces.

For a compact $\RCD(K, N)$ space $(X, \dist, \meas)$ whose essential dimension is $n \in [1, N] \cap \mathbb{N}$, it holds that for any $p \in [1, \infty)$, as $t \to 0^+$
  \begin{equation}\label{eq:bbgrcd}
    \frac{c(n)}{\omega_n}t\meas (B_{t^{1/2}}(x))g_t \to g, \quad \mathrm{in}\,L^p,
  \end{equation} 
where $g=g_{(X, \dist, \meas)}$ denotes the canonical Riemannian metric of $(X, \dist, \meas)$, see subsection \ref{sub:embedding} for the definition of $g$. %(see Theorem \ref{thm:convpull}). 
Moreover if in addition 
  \begin{equation}\label{eq:lower}%\footnote{In particular this condition implies that $\mathcal{H}^n$ and $\meas$ are equivalent.}
     \inf_{r \in (0,1), x \in X}\frac{\meas (B_r(x))}{r^n}>0
  \end{equation}  
holds, then we have a similar convergence result:
\begin{equation}\label{eq:asyrcd}
%\footnote{Note that it is proved in \cite{AHPT} that (\ref{eq:asyrcd}) holds under a weaker assumption than (\ref{eq:lower}). For example an $\RCD(n-1, n)$ space $([0, \pi], \dist_{[0, \pi]}, \sin^{n-1}t\di t)$ does not satisfy (\ref{eq:lower}), but (\ref{eq:asyrcd}) holds. However since we do not need this general convergence result in this paper.}
 c(n)t^{(n+2)/2}g_t \to \frac{\di \mathcal{H}^n}{\di \meas}g \quad\,\mathrm{in}\,L^p.
\end{equation}
It is worth pointing out that the finiteness of $p$ is sharp, that is, we can not replace $L^p$ by $L^{\infty}$ in general. For example any closed disc in $\mathbb{R}^n$ with the Lebesgue measure $\mathcal{L}^n$ gives such an example, see \cite[Remark 5.11]{AHPT}.
The convergence (\ref{eq:asyrcd}) shows us that the first principal term of the asymptotic behavior of $c(n)t^{(n+2)/2}g_t$ as $t \to 0^+$ is $\frac{\di \mathcal{H}^n}{\di \meas}g$ in the $L^p$-sense.
The main purpose is to discuss the second principal term. That is, the family of tensors indexed by $t$:
  \begin{equation}\label{eq:targetquotient}
     \frac{c(n)t^{(n+2)/2}g_t-\frac{\di \mathcal{H}^n}{\di \meas}g}{t}
  \end{equation}  
called the \textit{approximate Einstein tensor} of $(X, \dist, \meas)$ in this paper.
Let us ask when (\ref{eq:targetquotient}) is weakly asymptotically divergence free, that is, 
\begin{equation}\label{eq:main}
               \lim_{t \to 0^+}\int_X\left\langle \frac{c(n)t^{(n+2)/2}g_t-\frac{\di \mathcal{H}^n}{\di \meas}g}{t}, \nabla \omega\right\rangle \di \meas=0.
            \end{equation}  
holds for a large enough class of $1$-forms $\omega$. See Definition \ref{def:asy} for the precise definition of weakly asymptotically divergence free.

Our main result is stated as follows. Before stating it, recall that $D(\Delta_{H, 1})$ and $D(\delta)$ denote the domain of the Hodge Laplacian $\Delta_{H, 1}=\delta \di + \di \delta$ on $1$-forms defined in \cite{Gigli} and the domain of the adjoint operator $\delta=\di^*$ of the exterior derivative $\di$ on $1$-forms, respectively.

\begin{theorem}[``Weakly asymptotically divergence free'' characterizes the non-collapsed condition]\label{thm:main}
      Let $(X, \dist, \meas)$ be a compact $\RCD(K, N)$ space whose essential dimension is $n \in [1, N] \cap \mathbb{N}$.
      Then the following two conditions are equivalent:
  \begin{enumerate}
     \item  $(X, \dist, \meas)$ satisfies (\ref{eq:lower}) and (\ref{eq:main})
           for any $\omega \in D(\Delta_{H, 1})$ with $\Delta_{H, 1}\omega \in D(\delta)$.
                       
      \item $(X, \dist, \meas)$ is a $\RCD(K, n)$ space with
         \begin{equation}\label{eq:meashauss}
             \meas=\frac{\meas (X)}{\mathcal{H}^n(X)}\mathcal{H}^n.
         \end{equation}      
  \end{enumerate}
\end{theorem}

Since the space $\{\omega \in D(\Delta_{H, 1}); \Delta_{H, 1}\omega \in D(\delta)\}$ is dense in the space of $L^2$-1-forms, (\ref{eq:main}) can be interpreted as that the approximate Einstein tensor (\ref{eq:targetquotient}) is actually weakly asymptotically divergence free. See also appendix \ref{sec:app} (Corollary \ref{cor:spec}). Let us remark that (\ref{eq:meashauss}) implies that $(X, \dist, \mathcal{H}^n)$ is a non-collapsed $\RCD(K, n)$ space.   {It is worth pointing out that the compactness of $X$ in Theorem \ref{thm:main} cannot be dropped. See Example \ref{rem:example}.}

The following is a direct consequence of Theorem \ref{thm:main} which is also new (recall $\mathrm{vol}_f^g(A)=\int_Ae^{-f}\di \mathrm{vol}^g$):
\begin{corollary}\label{cor:weighted}
Let $(M^n, \dist^g, \mathrm{vol}_f^g)$ be a closed weighted Riemannian manifold. Then there exists a $G_f^g\in C^{\infty}((T^*)^{\otimes 2}M^n)$ called the weighted Einstein tensor such that the following expansion holds,  %Then denoting the expansion:
\begin{equation}
c(n)t^{(n+2)2}g_t=e^fg -\frac{2t}{3} G_f^g + O(t^2) \quad (t \to 0^+).
\end{equation}
%shows that
Moreover, $f$ is constant if and only if $G_f^g$ is divergence free with respect to $\mathrm{vol}^g_f$, that is,
\begin{equation}
\int_{M^n}\langle G_f^g, \nabla \omega \rangle \di \mathrm{vol}^g_f=0
\end{equation}
holds for any $\omega \in C^{\infty}(T^*M^n)$.
\end{corollary}
We will also provide a direct proof of this corollary with
%giving
 the explicit formula for $G_f^g$, see Proposition \ref{prop:directproof}.

%It is worth pointing out that for any compact $\RCD(K, N)$ space $(X, \dist, \meas)$ the set of $1$-forms that can be applied to the equality (\ref{eq:main}) is dense in $L^2(T^*X)$.That is, for any $\omega \in L^2(T^*X)$, there exists a sequence $\omega_i \in D(\Delta_{H, 1})$ such that $\Delta_{H, 1}\omega_i \in D(\delta)$ holds and that $\omega_i \to \omega$ in $L^2$. Moreover, it is dense in $H^{1,2}_C$, see Definition \ref{def:cov2}, and the discussion in section \ref{subsec:WADF}. 

It is worth noticing that although the left hand side of (\ref{eq:main}) converges as $t\to 0^+$, the approximate Einstein tensor itself (\ref{eq:targetquotient}) may not $L^2$-converge to a limit tensor in general. This is because lack of $L^2$ bounds, see section \ref{sec:stratified} for the explicit construction of a non-collapsed $\RCD(K,3)$ space with $K>1$ such that the $L^2$ norm of (\ref{eq:targetquotient}) tends to $+\infty$ as $t\to 0^+$. 
On the other hand, under assuming the uniform $L^2$ bound, we can prove that all limit tensors are actually divergence free as follows, which is an easy consequence of Theorem \ref{thm:main}.   
    
     %In particular for a non-collapsed $\RCD(K, n)$ space $(X, \dist, \mathcal{H}^n)$. there exist so many $\omega$ that can be applied to the equality (\ref{eq:main}) We have two proofs of this fact. 
	
    % The first one is to use the heat flow $h_{(H, 1), t}$ acting on $L^2$-$1$-forms associated to the Hodge energy of $\omega$:
   %\begin{equation}
     % \omega \mapsto \frac{1}{2}\int_X(|\dist \omega |^2 + |\delta \omega|^2)\di \mathcal{H}^n
  % \end{equation}
   % which is observed in \cite{Gigli} as the $L^2$ gradient flow of the Hodge energy according to the Brezis-Komura theory. Then for any $t \in (0, \infty)$ and any $L^2$-$1$-form $\omega$, we have $h_{(H, 1), t}\omega \in D(\Delta_{H, 1})$ with $\Delta_{H, 1}h_{(H, 1), t}\omega \in D(\Delta_{H, 1}) \subset D(\delta)$.

%The second proof is to use the spectral decomposition of $\Delta_{H, 1}$, which will be discussed in the appendix.  

%{\color{red}In connection with Theorem \ref{thm:main}, it is also interesting to establish B\'erard-Besson-Gallot's theorem for weighted closed Riemannian manifolds $(M^n, g, e^{-f}\mathrm{vol}_g)$ including the second principal term. If it is established, then we can also check the equivalence as in Theorem \ref{thm:main} directly.}\\

\begin{corollary}\label{cor:divfree}
  Let $(X, \dist, \mathcal{H}^n)$ be a compact non-collapsed $\RCD(K, n)$ space. If 
    \begin{equation}\label{eq:bounded}
       \sup_{0<t<1}\left\| \frac{c(n)t^{(n+2)/2}g_t-g}{t} \right\|_{L^2}<\infty
    \end{equation}     
holds, then any $G \in L^2((T^*)^{\otimes 2}(X, \dist, \mathcal{H}^n))$ that is a $L^2$-weak limit of some subsequence of   
  \begin{equation}
     \frac{c(n)t^{(n+2)/2}g_{t}-g}{t}
  \end{equation}
  as $t \to 0^+$ satisfies $G \in D(\nabla^*)$ with $\nabla^*G=0$, where $D(\nabla^*)$ denotes the domain of the divergence operator $\nabla^*$.
\end{corollary}
Applying Corollary \ref{cor:divfree} to a closed Riemannian manifold $(M^n, \dist^g, \mathrm{vol}^g)$ gives a new proof of (\ref{eq:divzero}) without using the Bianchi identity.

\smallskip\noindent
\textbf{Acknowledgement.}
Both authors are grateful to Igor Belegradek for valuable suggestions. Moreover they would like to thank the reviewer for his/her careful reading and for giving many valuable suggestions.
The first author acknowledges supports of {the Grant-in-Aid for Scientific Research
(B) of 21H00977}, the Grant-in-Aid for Scientific Research (B) of 20H01799 and 
the Grant-in-Aid for Scientific Research (B) of 18H01118.

\section{Heat kernel embedding}\label{sec:2}
The purpose of this section is to introduce our terminology minimally, assuming a bit of the knowledge of $\RCD$ theory. 
A triple $(X, \dist, \meas)$ is said to be a metric measure space if $(X, \dist)$ is a complete separable metric space and $\meas$ is a Borel measure with full support. For simplicity, we always assume that $X$ is not a single point.

\subsection{Definitions and the essential dimension}\label{subsec2}
Let us fix a metric measure space $(X, \dist, \meas)$.
Define the Cheeger energy $\Ch:L^2(X,\meas)\to [0,\infty]$ by
\begin{equation}\label{eq:defchee}
  \Ch(f):=\inf_{\|f_i-f\|_{L^2(X, \meas)}\to 0}
      \left\{ \liminf_{i\to\infty}\int_X{\rm lip}^2 f_i\di\meas:\ f_i\in\Lip_b(X,\dist)\cap L^2(X,\meas)
     \right\},
\end{equation}
where 
$$
{\rm lip}f(x) :=
\begin{cases}
\limsup\limits_{y\to x}\frac{|f(y)-f(x)|}{\dist(y,x)} &  \text{if $x \in X$ is not isolated},
\\
0 & \text{otherwise}
\end{cases}
$$
denotes the slope of $f$ at $x$. Then, the Sobolev space $H^{1,2}(X,\dist,\meas)$ is defined as the finiteness domain of $\Ch$.
By looking at the optimal sequence in \eqref{eq:defchee} one can identify a canonical object $|\nabla f|$, called the minimal relaxed slope,
which is local on Borel sets (i.e. $|\nabla f_1|=|\nabla f_2|$ $\meas$-a.e. on $\{f_1=f_2\}$) and provides an integral representation to $\Ch$, namely
$$
\Ch (f)=\int_X|\nabla f|^2\di\meas\qquad\forall f\in H^{1,2}(X,\dist,\meas).
$$
We are now in a position to introduce the $\RCD(K, N)$ spaces.
For any $K \in \mathbb{R}$ and any $N \in [1, \infty]$, a metric measure space $(X, \dist, \meas)$ is said to be an $\RCD(K, N)$ space if the following four conditions are satisfied.
\begin{enumerate}
  \item{(Volume growth)} There exist $x \in X$ and $C>1$ such that $\meas (B_r(x))\le Ce^{Cr^2}$ holds for any $r>0$.
  \item{(Inifinitesimally Hilbertian property)} $\Ch$ is a quadratic form. In particular thanks to \cite{AmbrosioGigliSavare14}, see also the first part of \cite{Gigli}, the function 
    $$
    \langle\nabla f_1,\nabla f_2\rangle:=\lim_{\epsilon\to 0}\frac{|\nabla (f_1+\epsilon f_2)|^2-|\nabla f_1|^2}{2\epsilon}
    $$
    provides a symmetric bilinear form on $H^{1,2}(X,\dist,\meas)\times H^{1,2}(X,\dist,\meas)$ with values in $L^1(X,\meas)$, and
    $$
    \mathcal{E} (f_1,f_2) := \int_X \langle \nabla f_1, \nabla f_2 \rangle \di \meas, \qquad \forall f_1, f_2 \in H^{1,2}(X,\dist,\meas) 
    $$
    defines a strongly local Dirichlet form.
  \item{(Sobolev-to-Lipschitz property)} Any $f \in H^{1, 2}(X, \dist, \meas)$ with $|\nabla f| \le 1$ for $\meas$-a.e. has an $1$-Lipschitz representative.
  \item{(Bochner inequality)} For any $f \in D(\Delta)$ with $\Delta f \in H^{1, 2}(X, \dist, \meas)$ we have
    \begin{equation}
      \frac{1}{2}\int_X|\nabla f|^2\Delta \phi \di \meas \ge \int_X\phi \left( \frac{(\Delta f)^2}{N}+\langle \nabla \Delta f, \nabla f\rangle +K|\nabla f|^2\right)\di \meas
    \end{equation}
    for any $\phi \in D(\Delta) \cap L^{\infty}(X, \meas)$ with $0 \le \phi \le 1$, $\Delta \phi \in L^{\infty}(X, \meas)$, where 
     \begin{align*}
       \mathcal{D}(\Delta) := \{ f \in H^{1,2}(X,\dist,\meas) \, : \, \, & \text{there exists} \, \, h \in L^2(X,\meas) \, \, \text{such that} \\
       &  \mathcal{E}(f,g)= - \int_X h g \di \meas \, \, \, \text{for all} \, \,  g\in H^{1,2}(X,\dist,\meas) \, \}
     \end{align*}
 and $\Delta f := h$ for any $f \in \mathcal{D}(\Delta)$.
\end{enumerate}
See \cite[Sec.12]{AmbrosioMondinoSavare} and \cite[Thm.7 and Sec.4]{ErbarKuwadaSturm}.
It is worth pointing out that if $N<\infty$, then for any $\RCD(K, N)$ space $(X, \dist, \meas)$ and any locally Lipschitz function $f$ on $X$ belonging to $H^{1, 2}(X, \dist, \meas)$, we have 
\begin{equation}\label{eq:lipcheeger}
|\nabla f|(x)=\mathrm{lip}f(x),\quad \mathrm{for}\,\, \meas-a.e. \,\,x \in X
\end{equation}
because of \cite[Thm.6.1]{Cheeger}, the Bishop-Gromov inequality \cite[Thm.5.31]{LottVillani}, \cite[Thm.2.3]{Sturm06b} and the Poincar\'e inequality \cite[Thm.1]{Rajala}.
For any $k \geq 1$, we denote by $\mathcal{R}_k$ the $k$-dimensional regular set  of $(X, \dist, \meas)$, 
namely the set of points $x \in X$ such that $(X, r^{-1}\dist, \meas (B_{r}(x))^{-1}\meas, x)$ pointed measured Gromov-Hausdorff converge to $(\mathbb{R}^k, \dist_{\mathbb{R}^k}, \omega_k^{-1}\mathcal{L}^k,0_k)$ as $r \to 0^+$, where $B_r(x)$ denotes the open ball centered at $x$ with the radius $r$.
It is proved in \cite[Thm.0.1]{BrueSemola}
that if $(X,\dist,\meas)$ is an $\RCD (K,N)$ space with $N<\infty$, then there exists a unique integer $n\in [1,N]$, denoted by $\dim_{\dist,\meas}(X)$, called the essential dimension of $(X, \dist, \meas)$, such that
 \begin{equation}\label{eq:regular set is full}
   \meas(X\setminus \mathcal{R}_n\bigr)=0.
 \end{equation}

\subsection{The heat kernel}
 Throughout this paper the parameters $K\in\mathbb{R}$ and $N \in [1, \infty)$ will
 be kept fixed. Let us fix a $\RCD(K, N)$ space $(X, \dist, \meas)$.
 Then thanks to \cite[Prop.2.3]{Sturm95} and \cite[Cor.3.3]{Sturm96}, the (H\"older continuous) heat kernel $p:X\times X \times (0, \infty) \to (0, \infty)$ of $(X, \dist, \meas)$ is well-defined by satisfying
\begin{equation}
   h_tf=\int_Xp(x, y, t)f(y)\di \meas(y),\quad \forall f \in L^2(X, \meas),
\end{equation}
 where $h_t:L^2(X, \meas) \to L^2(X, \meas)$ is the heat flow associated with the Cheeger energy $\Ch$. 
The sharp Gaussian estimates on this heat kernel proved in \cite[Thm.1.2]{JiangLiZhang} state that
for any $\epsilon>0$, there exist $C_i:=C_i(\epsilon, K, N)>1$ for $i=1,\,2$, depending only on $K$, $N$ and $\epsilon$, such that 
 \begin{equation}\label{eq:gaussian}
  \small \frac{C_1^{-1}}{\meas (B_{\sqrt{t}}(x))}\exp \left(-\frac{\dist^2 (x, y)}{(4-\epsilon)t}-C_2t \right) \le p(x, y, t) \le \frac{C_1}{\meas (B_{\sqrt{t}}(x))}\exp \left( -\frac{\dist^2 (x, y)}{(4+\epsilon)t}+C_2t \right)
\end{equation}
for all $x,\, y \in X$ and any $t>0$, where from now on we state our inequalities with the H\"older continuous representative. Combining (\ref{eq:gaussian}) with the Li-Yau inequality \cite[Cor.1.5]{GarofaloMondino}, \cite[Thm.1.2]{Jiang15}, 
we have a gradient estimate \cite[Cor.1.2]{JiangLiZhang}:
 \begin{equation}\label{eq:equi lip}
  |\nabla_x p(x, y, t)|\le \frac{C_3}{\sqrt{t}\meas (B_{\sqrt{t}}(x))}\exp \left(-\frac{\dist^2(x, y)}{(4+\epsilon) t}+C_4t\right)
  \qquad\text{for $\meas$-a.e. $x\in X$}
\end{equation}
for any $t>0$, $y\in X$, where $C_i:=C_i(\epsilon, K, N)>1$ for $i=3,\,4$.
\subsection{Embedding}\label{sub:embedding}
Throughout the subsection, we only refer to \cite{Gigli} for the details of tensor fields on $\RCD$ spaces, including:
\begin{itemize}
  \item the spaces of all $L^p$-$1$-forms, of all $L^p$-tensor fields of type $(0, 2)$, denoted by $L^p(T^*(X, \dist, \meas))$, $L^p((T^*)^{\otimes 2}(X, \dist, \meas))$, respectively;
  \item the pointwise scalar product $\langle S, T\rangle$ for two tensor fields of the same type. 
\end{itemize}
Note that one of the canonical operators, the so-called exterior derivative for functions, $\di: H^{1,2}(X, \dist, \meas) \to L^2(T^*(X, \dist, \meas))$ satisfy $|\di f|=|\nabla f|$ for $\meas$-a.e. $x \in X$.

Let us fix a compact $\RCD(K, N)$ space $(X, \dist, \meas)$ with $n=\dim_{\dist, \meas}(X)$. Then thanks to the Bishop-Gromov inequality and the Poincar\'e inequality, we know that the canonical inclusion $H^{1, 2}(X, \dist, \meas) \hookrightarrow L^2(X, \meas)$ is a compact operator by \cite[Thm.8.1]{HK}. In particular the heat kernel $p$ of $(X, \dist, \meas)$ has the following expansion:
\begin{equation}\label{eq:expansion1}
  p(x,y,t) = \sum_{i \ge 0} e^{- \lambda_i t} \phi_i(x) \phi_i (y) \qquad \text{in $C(X\times X)$}
\end{equation}
for any $t>0$ and
 \begin{equation}\label{eq:expansion2}
   p(\cdot,y,t) = \sum_{i \ge 0} e^{- \lambda_i t} \phi_i(y) \phi_i \qquad \text{in $H^{1,2}(X,\dist,\meas)$}
 \end{equation}
for any $y\in X$ and $t>0$,
where 
 \begin{equation}
   0=\lambda_0 < \lambda_1 \le \lambda_2 \le \cdots \to \infty
 \end{equation}
denote the discrete nonnegative spectrum of $-\Delta$ counted with multiplicities, and $\phi_0, \phi_1, \ldots $ are the corresponding (H\"older continuous) eigenfunctions with $\|\phi_i\|_{L^2}=1$. 
Combining (\ref{eq:expansion1}) and (\ref{eq:expansion2}) with (\ref{eq:equi lip}), we know that $\phi_i$ is Lipschitz, in fact, it holds that
 \begin{equation}\label{eq:eigenest}
   \|\phi_i\|_{L^\infty} \leq C_5 \lambda_i^{N/4}, \qquad \| \nabla \phi_i \|_{L^\infty} \leq C_5 \lambda_i^{(N+2)/4}, \qquad \lambda_i \ge C_5^{-1}i^{2/N},
 \end{equation}
where $C_5:=C_5(\mathrm{diam} (X, \dist), K, N)>0$.

It is proved in (the proof of) \cite[Prop.4.1]{AHPT} by using (\ref{eq:expansion1}) that for any $t>0$ the map $\Phi_t:X \to L^2(X, \meas)$ defined by
 \begin{equation}
   \Phi_t(x)(y):=p(x, y, t)
 \end{equation}
is a topological embedding. Then since (\ref{eq:equi lip}) proves that $\Phi_t$ is Lipschitz, we can define the pull-back metric $\Phi_t^*g_{L^2}$, denoted by $g_t$, by
 \begin{equation}\label{DefPullBack}
   g_t:=\sum_ie^{-2\lambda_it}\di \phi_i \otimes \di \phi_i,\quad \mathrm{in}\,\,L^{\infty}\left((T^*)^{\otimes 2}(X, \dist, \meas)\right),
  \end{equation}
%where $L^{p}((T^*)^{\otimes 2}(X, \dist, \meas))$ denotes the space of all $L^{p}$-tensors of type $(0, 2)$ on $(X, \dist, \meas)$.
 Note that in \cite{AHPT}, the equality of (\ref{DefPullBack}) is stated in $L^2((T^*)^{\otimes 2}(X, \dist, \meas))$, however, thanks to (\ref{eq:equi lip}), this holds in $L^{\infty}((T^*)^{\otimes 2}(X, \dist, \meas))$, and that there exists a unique $g=g_{(X, \dist, \meas)} \in L^{\infty}((T^*)^{\otimes 2}(X, \dist, \meas))$, called the Riemannian metric of $(X, \dist, \meas)$, such that $\langle g, \di f_1 \otimes \di f_2\rangle (x)=\langle \nabla f_1, \nabla f_2\rangle (x)$ holds for $\meas$-a.e. $x \in X$.

A convergence result proved in \cite[Thm.5.10]{AHPT} states that 
 \begin{equation}\label{eq:ahpt}
   \frac{c(n)t}{\omega_n}\meas (B_{\sqrt{t}}(x))g_t \to g,\quad \mathrm{in}\,\,L^p\left((T^*)^{\otimes 2}(X, \dist, \meas)\right),
 \end{equation}
for all $p \in [1, \infty)$.
In particular if $\meas \le C\mathcal{H}^n$ holds for some $C>0$, since
\begin{equation}
\frac{\meas (B_r(x))}{\omega_nr^n} \to \frac{\di \meas}{\di \mathcal{H}^n}(x), \quad \mathrm{for}\,\,\meas-a.e.\, x\in X
\end{equation}
as $r \to 0^+$ which is proved in \cite[Thm.4.1]{AmbrosioHondaTewodrose} as a more general result, then combining the dominated convergence theorem with (\ref{eq:ahpt}) yields 
\begin{equation}\label{eq:firstpricipal}
c(n)t^{(n+2)/2}g_t \to \frac{\di \mathcal{H}^n}{\di \meas} g,\quad \mathrm{in}\,\,L^p\left((T^*)^{\otimes 2}(X, \dist, \meas)\right).
\end{equation}
See \cite[Thm.5.15]{AHPT} for a more general statement.
%\footnote{I will replace $c(n)$ by explicit one later}

\section{Second principal term in weighted Riemannian case}\label{sec:3}

{

Let us start this section by discussing relationships between the notions that appeared in the previous section and smooth objects.
We fix a smooth weighted complete Riemannian manifold $(M^n, g, \mathrm{vol}_{f}^g)$, where $f \in C^{\infty}(M^n)$, and for any Borel subset $A$ of $M^n$,
\begin{equation}
\mathrm{vol}_{f}^g(A) :=\int_{A}e^{-f}\di \mathrm{vol}^g.
\end{equation}
%{\color{blue}Note that it is enough to consider the case when $(M^n, g)$ is the interior of a compact Riemannian manifold with smooth boundary\footnote{{\color{red} Why do we need boundary? I thought it could be something like when it is enough to consider the case when $(M^n,g)$ is a closed manifold with finitely many. points removed}} for our purpose (see the proof Proposition \ref{prop:AlmostSmoothBBG}).}
Recall that $(M^n, \dist^{g}, \mathrm{vol}_{f}^g)$ is an $\RCD(K, N)$ space if and only if $n \ge N$, and
\begin{equation}\label{eq:be}
\mathrm{Ric}^g+\mathrm{Hess}_f^g-\frac{\di f \otimes\di f}{N-n} \ge Kg,
\end{equation}
where
if $n=N$ holds, then (\ref{eq:be}) is understood as that $f$ is constant and that $\mathrm{Ric}^g \ge Kg$ holds, see \cite[Prop.4.21]{ErbarKuwadaSturm}. In particular if $M^n$ is closed, then for any $N >n$ there exists $K \in \mathbb{R}$ such that $(M^n, \dist^{g}, \mathrm{vol}_{f}^g)$ is an $\RCD(K, N)$ space whose essential dimension is trivially equal to $n$.
This setting will be discussed in the following subsections. 

%We first discuss the \textit{Dirichlet Laplacian} on $(M^n, \dist^g, \mathrm{vol}^g_f)$ without assuming completeness of $(M^n, \dist^g)$, which will play a key role in section \ref{sec:stratified} to find an example which shows that Theorem \ref{thm:main} is sharp in some sense. To be precise, let us clarify the meaning of the Dirichlet Laplacian and the Dirichlet heat kernel $p_f$ of $(M^n, \dist^g, \mathrm{vol}^g_f)$ for the reader's convenience.

%Let $H^{1, 2}_0(M^n, \dist^g, \mathrm{vol}^g_f), {\color{blue}H^{2,2}_0(M^n, \dist^g, \mathrm{vol}^g_f)}$ denote the completion of $C^{\infty}_c(M^n)$ with respect to the $H^{1,2}$, {\color{blue}$H^{2,2}$}-norm respectively, the \textit{Dirichlet Laplacian}:

Let us discuss the Laplacian $\Delta$ on a metric measure space $(M^n, \dist^g, \mathrm{vol}_{f}^g)$ as defined in the subsection \ref{subsec2}. This coincides with the weighted Laplacian $\Delta_f^g$ for any $\phi \in C^{\infty}(M^n) \cap D(\Delta)$ namely;
\begin{equation}\label{eq:witten}
\Delta^g_f\phi:=\mathrm{tr}(\mathrm{Hess}_{\phi}^g)-g(\nabla f, \nabla \phi)
\end{equation}
because we see
\begin{equation}
\int_Xg(\nabla \psi, \nabla \phi)\di \mathrm{vol}^g_f=-\int_X\psi \Delta^g_f\phi \di \mathrm{vol}^g_f,\quad \forall \psi \in C_c^{\infty}(M^n)
\end{equation}
which implies the coincidence between $\int \Delta^g_f\phi\di \mathrm{vol}^g_f$ and $\int \Delta \phi \di \mathrm{vol}^g_f$ as measures.
%{is a self-adjoint operator in $L^2(M^n,\mathrm{vol}^g_f)$ with domain $H^{2, 2}_0(M^n, \dist^g, \mathrm{vol}^g_f)$ , see \cite[Section 4.2]{Grig}.} 
Then the heat flow $h_{f, t}$ on the metric measure space $(M^n, \dist^g, \mathrm{vol}_{f}^g)$ is uniquely determined as follows: for any $\phi \in L^2(M^n, \mathrm{vol}_f^g)$, the map $t \mapsto h_{f, t}\phi \in L^2(M^n, \mathrm{vol}^g_f)$ is smooth on $(0, \infty)$ with $h_{f, t}\phi \in C^{\infty}(M^n) \cap D(\Delta)$,
\begin{equation}
\frac{\di}{\di t}h_{f, t}\phi =\Delta_f^g h_{f, t}\phi \quad \mathrm{in}\,L^2(M^n, \mathrm{vol}^g_f),
\end{equation}
and $h_{f, t}\phi \to \phi$ in $L^2(M^n, \mathrm{vol}^g_f)$ as $t \to 0^+$ (see for instance
%The existence of such a semigroup is known given that $\Delta^g_f$ is self-adjoint, see 
\cite[Thm. 4.9]{Grig}). 
Finally the heat kernel $p_f$ is uniquely determined by being smooth and satisfying
\begin{equation}
h_{f, t}\phi(x)=\int_{M^n}p_f(x, y, t)\phi(y)\di \mathrm{vol}^g_f(y),\quad \forall \phi \in L^2(M^n, \mathrm{vol}^g_f),\,\,\,\forall x \in M^n.
\end{equation}
%Then the Riesz representation theorem yields that for any $t \in (0, \infty)$ and any $x \in M^n$, there exists a unique $p_{t, x} \in L^2(M^n, \mathrm{vol}^g_f)$ such that 
%\begin{equation}
%h_{f, t}\phi (x)=\int_{M^n}p_{t, x}(y)\phi (y)\di \mathrm{vol}^g_f
%\end{equation}
%holds for any $\phi \in L^2(M^n, \mathrm{vol}^g_f)$. Then the \textit{heat kernel} $p_f(x, y, t)$ is defined by 
%\begin{equation}
%p_f(x, y, t):=\int_{M^n}p_{t/2, x}(z)p_{t/2, y}(z)\di \mathrm{vol}^g_f(z)
%\end{equation}
%which is { unique} and smooth on $M^n \times M^n \times (0, \infty)$, see \cite[Def. 7.12, Rmk. 7.14]{Grig}.
It is worth pointing out that similar observations above are also justified in the case when $(M^n, g, \mathrm{vol}^g_f)$ is the interior of a smooth weighted compact Riemannian manifold with smooth boundary after replacing the Laplacian, the heat flow, and the heat kernel by the \textit{Dirichlet}'s ones, respectively. We omit the details, where this will play a role to find an example which shows that Theorem \ref{thm:main} is sharp in some sense (see the proof of Proposition \ref{prop:AlmostSmoothBBG}).

}

From now on, let $(r,\xi^1,\xi^2,\ldots, \xi^{n}):=(r,\xi)$ be the normal coordinates around $x \in M^n$, and $g(r,\xi)$ be the Riemannian metric at the point $(r,\xi)$ in the normal coordinates. We introduce the following elementary lemma which will play a role later. 
{ In the following lemma and in the sequel, we know from the proofs, all remainder terms of the form $O(t^k)=O_{f, g}(t^k)$ on $(M^n, g)$ as $t\to 0^+$ have smooth coefficients and depend only on the metric $g$ and the weight $f$.}

  \begin{lemma}\label{lem:volasym}
    For any $x\in M^n$ we have the following asymptotic expansion as $r \to 0^+$
    \begin{equation}\label{eq:volasymp}
	  \vol_{f}^g(B_r(x))=\omega_n r^ne^{-f(x)}\left(1-\frac{\mathrm{Scal}^g+3\Delta^g f-3|\nabla f|^2}{6(n+2)}r^2 +O(r^3)\right),
    \end{equation}

   Moreover, { the remainder in the asymptotic expansion} (\ref{eq:volasymp}) {has a} uniform {bound} for any compact subset $K\subset M^n$ in the sense that 
    \begin{equation}\label{eq:unifasym}
      \resizebox{.85\hsize}{!}{$\sup\limits_{x \in K, r<1}r^{-3-n}\left|\vol_{f}^g(B_r(x))-\omega_n r^ne^{-f(x)}\left(1-\frac{\mathrm{Scal}^g+3\Delta^g f-3|\nabla f|^2}{6(n+2)}r^2\right) \right|<\infty$}.
    \end{equation}
  \end{lemma}
  
\begin{proof}
	 Recall that for any unit vector $v\in T_xM$ and any geodesic $\gamma$ from $x$ with $\dot{\gamma}(0)=v$, it follows from Taylor expansion that
   \begin{align}
	    \sqrt{\det g(\gamma(t))} &=1-\frac{\mathrm{Ric}^g(v, v)}{6}t^2+O(t^3), \label{volexp} \\
	e^{-f(\gamma(t))+f(x)}   &=1-\langle\nabla f(x),v\rangle t+\frac{1}{2}\left(-\mathrm{Hess}_f^g(v,v)+|\langle\nabla f(x),v\rangle|^2\right)t^2+O(t^3).\label{eq:densityexp}
   \end{align}
Thus we have
    \begin{align*}
       \vol_{f}^g(B_r(x))&=\int_0^r\int_{S^{n-1}}\left(1-\frac{\mathrm{Ric}^g_{ij}}{6}\xi^i\xi^jt^2+O(t^3)\right)\cdot \\
                &\quad\left[1-(\nabla f)_i\xi^it+\frac{1}{2}\left(-\mathrm{Hess}_{f,ij}^g+(\di f\otimes\di f)_{ij}\right)\xi^i\xi^j t^2+O(t^3)\right]e^{-f(x)}t^{n-1}\di \xi \di t\\
                %&=e^{-f(x)}\int_0^r\int_{S^{n-1}}t^{n-1}-\frac{1}{6}\left(\mathrm{Ric}^g_{ij}+3\mathrm{Hess}_{f,ij}-3(\dist f\otimes\di f)_{ij}\right)\xi^i\xi^jt^{n+1}+O(t^{n+2})\di \xi \di t\\
                &=\omega_n r^ne^{-f(x)}\left(1-\frac{\mathrm{Scal}^g+3\Delta^g f-3|\nabla f|^2}{6(n+2)}r^2 +O(r^3)\right)
    \end{align*}
    as desired, where $\mathrm{Hess}_f^g$, $\di f\otimes\di f$, $\nabla f$ and $\mathrm{Ric}^g$ are all evaluated at $x$. By expanding the left hand side of (\ref{volexp}) and (\ref{eq:densityexp}) to the $t^3$ or higher order terms, we can infer that the coefficients involves the derivatives of the Riemannian curvature tensor, and the derivatives of $f$, respectively. Since they are all smooth objects, they are uniformly bounded on any compact set $K$, so the uniform bound (\ref{eq:unifasym}) follows.
  \end{proof}

%in fact, we will see in Section \ref{sec:stratified} that the expansion of $g_{f,t}$ is local. 
%We follow arguments in \cite[Section 3.2]{Rosenberg}. First let us introduce some notations which are compatible with that used in \cite{Rosenberg}.

\subsection{The weighted heat kernel expansion}\label{subsec:setup}
{From now on we assume that $M^n$ is closed. 
Denote by $\mathrm{inj}^g$ the injectivity radius of $(M^n, g)$,} consider 
$$
V=\{(x,y)\in M^n \times M^n:\dist^g(x,y)<{\mathrm{inj}^g/2}\}.
$$
Fix $k \in \mathbb{Z}_{>0}$, let us find $u_j{=u_{j, k}}\in C^{\infty}(V)$, $j={0,} 1, 2, \ldots, k$ such that
\begin{equation}\label{eq:recursion}
   \left(\Delta^g_{f,x}-\partial_t\right)S_k=\frac{1}{(4\pi t)^{n/2}}\exp \left(-\frac{\dist^g(x,y)^2}{4t}\right)\cdot t^k\cdot\Delta^g_{f,x} \left(u_ke^A\right), \quad \forall (x, y) \in V
\end{equation}
holds, where $A=A(x,y)= \frac{f(x)+f(y)}{2}$ and
   \begin{equation}\label{S_k}
	S_k(x,y,t)=\frac{1}{(4\pi t)^{n/2}}\exp \left(-\frac{\dist^g(x,y)^2}{4t}+A(x, y)\right)\cdot \sum_{j=0}^kt^ju_j(x,y).
    \end{equation}
{ In fact}, the desired functions $u_j$ are { uniquely} obtained as follows, {in particular $u_{j, k}$ is independent of $k$.} 
%where we see the independence of $u_{j, k}$ with respect to $k$.
    \begin{lemma}\label{solveU}
     We have
       \begin{equation}
       \resizebox{.9\hsize}{!}{$ \begin{split}\label{u_jrecur}
            u_0(x,y)=&D^{-\frac{1}{2}}(y)\\
           u_j(x,y)=&\dist^g(x, y)^{-j}D^{-1/2}(y)	\left[ \int_0^{\dist^g(x,y)} D^{1/2}(\gamma(s))\Delta^g_{\gamma(s)}u_{j-1}(x,\gamma(s))s^{j-1}\di s\right.\\
                      &\left. +\int_0^{\dist^g(x, y)}D^{1/2}(\gamma(s))\left(\frac{1}{2}\Delta^g f(\gamma(s))-\frac{1}{4}|\nabla f(\gamma(s))|^2\right)u_{j-1}(x,\gamma(s))s^{j-1}\di s\right]
          \end{split}$}
         \end{equation}
           where $j\ge 1$ and $\gamma$ is the unit speed minimal geodesic from $x$ to $y$, and $D(y)=\frac{\sqrt{\det g(r,\xi)}}{\dist^g(x, y)^{n-1}}$ which is the volume density at $y$ in normal coordinates $(r,\xi)$ around $x$.
     \end{lemma}
     
  \begin{proof}
%A direct computation shows
%\begin{equation}
	%\begin{split}
		%\partial_t S_k(x, y)&=\left(\frac{\dist^g(x, y)^2}{4t^2}-\frac{n}{2t}\right)S_k(x, y, t)+Ge^{A(x, y)}(u_1(x, y)+tu_2(x, y)+\cdots+kt^{k-1}u_k(x, y))\\
		%\Delta_f^g S_k&=\Delta^g(GU_k)e^A+2\langle \nabla(GU_k),\nabla e^A\rangle+GU_k\Delta^g(e^A)-\langle\nabla f,\nabla (GU_k e^A)\rangle
	%\end{split}
%\end{equation}
   From (\ref{eq:recursion}) with (\ref{eq:witten}), we obtain that (\ref{eq:recursion}) is equivalent to
     \begin{equation}\label{eq:recursion_mid}
      \resizebox{.9\hsize}{!}{$\begin{split}
         0&=\dist^g(x, y)\partial_r u_0+\frac{\dist^g(x, y)}{2}\frac{\partial_r D}{D}u_0\\
	     0&=\dist^g(x, y)\partial_r u_j+\left(j+\frac{\dist^g(x, y)}{2}\frac{\partial_r D}{D}\right)u_j-\Delta^g u_{j-1}-\left(\frac{1}{2}\Delta^g f-\frac{1}{4}|\nabla f|^2\right)u_{j-1} 
      \end{split}$}
     \end{equation}
where $ j\geq 1$ and $r=\dist^g(x, y)$ and $\partial_r$ is the radial derivative from $x$, we give a sketch of this computation.
 %see details in \cite[Section 3.2]{Rosenberg}. 
 Solve the first equation of (\ref{eq:recursion_mid}), to get $u_0(x,y)=C(\xi)D^{-\frac{1}{2}}(y)$, note that $u_0(x,x)=1$, so $C(\xi)=1$, then we get the first equality of (\ref{u_jrecur}). To yield the second equation of (\ref{u_jrecur}), we first solve the corresponding homogeneous equation of the second equation of (\ref{eq:recursion_mid}), which is  
    \begin{equation}
 	\dist^g(x, y)\partial_r u_j+\left(j+\frac{\dist^g(x, y)}{2}\frac{\partial_r D}{D}\right)u_j=0,
    \end{equation}
then we use the method of variation of parameters to finish the computation. 
\end{proof}

%Fix a $x\in M$, for any $y\in \exp_x(B_{\epsilon}(0))$, set $r=\dist^g(x,y)$, in geodesic polar coordinates, there is a unit vector $\xi=(\xi^1,\xi^2,\ldots, \xi^{n})\in T_xM^n$ such that $y=\exp_x(r\xi)$. It follows 
%$$
%\Delta_f^g=\partial_r^2+\frac{\partial_r \sqrt{\det g}}{\sqrt{\det g}}\partial_r+\Delta^g_f|_{S_r^{n-1}},\quad \nabla=\partial_r+\nabla|_{S_r^{n-1}}
%$$
%Let  $U_k=u_0+tu_1+\cdots +t^ku_k$ be functions defined on $V_{\epsilon}$, and 
%\begin{equation*}
	%\begin{split}
		       %G:V_{\epsilon}\times (0,+\infty)&\rightarrow \setR\\
		%G(x,y,t)&=\frac{1}{(4\pi t)^{n/2}}e^{-\frac{d^2(x,y)}{4t}}
	%\end{split}
%\end{equation*}
%be the standard heat kernel on $\setR^n$. We modify $G$ to a function $S_k:V_{\epsilon}\times (0,+\infty)\rightarrow \setR$ by 
%\begin{equation}\label{S_k}
	%S_k(x,y,t)=G(x,y,t)e^{A(x,y)}\left(u_0(x,y)+tu_1(x,y)+\cdots t^ku_k(x,y)\right)
%\end{equation}
%for every nonnegative integer $k$. where $u_j\in C^{\infty}(V_{\epsilon})$, $j=1, 2, \ldots, k$, are to be chosen so that

%This ensures that $S_k$ vanishes to the order $t^{k-n/2}$ as $t\rightarrow 0^+$.

We follow \cite{Chavel} closely. 
%and point out that for the local computation for the short time expansion of heat kernel, completeness and compactness are not needed, this is already noticed by Cheeger-Yau for manifolds without weight, moreover, they showed for two possibly different extensions of the Laplace operator, the short time near diagonal expansion for corresponding heat kernels are the same, see  \cite{CheegerYau}. 
The first step is to extend $S_k$ to whole $M^n\times M^n$ by multiplying a cut-off function $\phi(x,y)\in C^{\infty}(M^n\times M^n)$ so that for each $y\in M^n$, $\phi(x,y)=0$ on $M^n \setminus B_{{\mathrm{inj}^g/2}}(y)$, $\phi(x,y)=1$ on $B_{{\mathrm{inj}^g/4}}(y)$ and $0\le\phi(x,y)\le 1$. Let
\begin{equation}\label{eq:cutoff}
H_k(x,y,t):=\phi(x,y)S_k(x,y,t) \in C^{\infty}\left(M^n\times M^n\times(0,\infty)\right).
\end{equation}

%
%To complete the desired construction for the heat kernel, we introduce the (weighted) convolution $F*H$ for $F,H\in C^{0}(M)$:
%$$
%F*H(x,y,t)=\int_0^t\int_M F(x,z,s)H(z,y,t-s)\di \mathrm{vol}_f^g(z)ds, \quad \forall F, H \in C^0(M^n\times M^n\times (0,\infty))
%$$
%and denote $H^{*j}=H*H*\cdots *H$ for $j$-fold convolution. The proof in \cite[section 3.3]{Rosenberg} implies
%\begin{proposition}
%	For any $k>\frac{n}{2}+2$,  $p_f=H_k-H_k*F_k$ is the global heat kernel, where 
%   \begin{equation}
%     F_k=\sum_{j\geq 0}(-1)^{j+1}((\partial_t-\Delta^g)H_k)^{*j}.
%   \end{equation}
%    In particular  we have
 %  \begin{equation}\label{esti F}
 %    \|t^{n/2-k}F_k(\cdot,\cdot,t)\|_{L^{\infty}(M^n\times M^n \times (0, t_0])}<\infty, \quad \forall t_0>0.
%   \end{equation}
 % \end{proposition}
%  
  The following properties are known for $H_k$:
 
\begin{enumerate}
\item {\label{eq:regularityH} $(\partial_t-\Delta_f^g)H_k$ extends to $t=0$ and} $(\partial_t-\Delta_f^g)H_k\in C^{\ell}(M^n\times M^n\times [0,\infty))$ for integer $\ell<k-\frac{n}{2}$
%for any integer $\ell <k-\frac{n}{2}$
;
\item \label{eq:short timeH} $H_k(x,y,t)\to \delta_y(x)$ for all ${x, }y\in M^n$ {as $t \to 0^+$}.	
\end{enumerate}

See \cite[Lem 1, Chap. VI Sec. 4]{Chavel}, and \cite[Lem. 3.18]{Rosenberg}. 
%Note that in both references compactness and completeness of $M^n$ are assumed, but it is irrelevant here since the computation is local. 
%{\color{red}In particular it implies that $H_k$ is a parametrix of $p_f$ when $k>\frac{n}{2}+2$.}
%A similar argument as in \cite[p.467-468]{CheegerYau} shows that $H_k$ is a parametrix of $p_f$ when $k>\frac{n}{2}+2$.
%{\color{red}when $k>n/2$ and $|f|, |\nabla f|, |\mathrm{Hess}^g_f|\in C_b(M^n)$}.\footnote{Is the later condition, $|f|, |\nabla f|, |\mathrm{Hess}^g_f|\in C_b(M^n)$,  really needed to justify the following theorem? Let us check the necessarity again. It is better to explain where these conditions play roles in the proof if so.}  
%In the unweighted case, that is, when the weight function $f$ is a constant, the global heat kernel $p_f$ as a smooth fundamental solution is obtained from this parametrix by the argument in \cite[p.467-468]{CheegerYau}. When $f$ is not a constant, under the assumption that $|f|, |\nabla f|, |\mathrm{Hess}^g_f|\in L^{\infty}(M^n)$, the proof in \cite{CheegerYau} can be directly applied to show the existence of $p_f$, we omit the details here. 

We are now in position to establish the asymptotic expansion of $p_f$. It is worth pointing out that (\ref{u_1}) is computed in \cite{MS} with a slightly different normalization of the heat kernel. 

We introduce the (weighted) convolution $F*H$ for $F,H\in C^{0}(M^n)$:
$$
F*H(x,y,t)=\int_0^t\int_M F(x,z,s)H(z,y,t-s)\di \mathrm{vol}_f^g(z)ds, \quad \forall F, H \in C^0(M^n\times M^n\times [0,\infty))
$$
and denote $H^{*j}=H*H*\cdots *H$ for $j$-fold convolution. Let 
 \begin{equation}
     F_k=\sum_{j\geq 0}(-1)^{j+1}((\partial_t-\Delta^g_f)H_k)^{*j}.
 \end{equation}

{
 Although the following are quite standard, for the reader's convenience, we show some similar estimates as in \cite[p.152 Lemma 1]{Chavel}. First note that
\begin{equation}\label{eq:S_k esti}
	(\partial_t-\Delta_f^g)H_k= \phi(\partial_t-\Delta_f^g)S_k-2\langle \nabla \phi,\nabla S_k\rangle-S_k\Delta_f^g\phi.  
\end{equation}
Recalling (\ref{eq:recursion}), we see that the first term on the RHS of the above equation is bounded by $C(f,g) t^{k-\frac{n}{2}}$. The rest 2 terms decay exponentially as $t\to 0^+$ because $\nabla \phi$ and $\Delta_f^g\phi$ vanishes near the diagonal and the Gaussian term gives the exponential decay away from the diagonal, thus 
%Namely there exists $t_0>0$\footnote{SH: It seems to me that we can choose $t_0=1$.} 
 for any $t\in [0,1]$ the last 2 terms are bounded by $C(f,g)t^{k-\frac{n}{2}}$. Thus it follows that 
    \begin{align}\label{eq:esti H}
       \left\|(\partial_t-\Delta_f^g)H_k(\cdot,\cdot,t)\right\|_{L^\infty(M^n\times M^n)}&\le C(f, g)t^{k-n/2}, \quad \forall t\in (0,1].
  \end{align}
%Fix such $t_0>0$, and 
Let 
\begin{equation}
	(\partial_t-\Delta_f^g)H_k(x,y,t)=t^{k-\frac{n}{2}}\exp\left(-\frac{\dist^g(x,y)^2}{4t}+A(x,y)\right)G_k(x,y,t).
\end{equation}
 We see from \eqref{eq:S_k esti} that $G_k\in C^{\infty}(M^n\times M^n\times [0,\infty))$. Set 
$$
B:= \sup_{M^n\times M^n\times [0,1]} |G_k|,
%\qquad B:=At_0^{k-\frac{n}{2}},
$$
we have for any $t\in (0,1]$,
\begin{align}\label{eq:convesti}
	|((&\partial_t -\Delta_f^g)H_k)^{*2}| (x,y,t)\notag\\
	&\le\int_0^{t}\int_{M^n} s^{k-\frac{n}{2}}(t-s)^{k-\frac{n}{2}}|G_k(x,z,s)G_k(z,y,t-s)| e^{-\frac{\dist^g(x,z)^2}{4s}}e^{-\frac{\dist^g(z,y)^2}{4(t-s)}}e^{\frac{f(x)+f(y)}{2}}\di \vol^g(z)\di s \notag\\
	&\le B^2\vol^g(M) e^{-\frac{\dist^g(x,y)^2}{4t}+A(x,y)}\int_0^{t} s^{k-\frac{n}{2}}(t-s)^{k-\frac{n}{2}}\di s\notag \\
	&\le \frac{B^2\vol^g(M)t^{k-\frac{n}{2}+1}}{k-\frac{n}{2}+1}e^{-\frac{\dist^g(x,y)^2}{4t}+A(x,y)}.
\end{align}
Here, we have used $\vol^g_f=e^{-f}\vol^g$. Using induction, one can show similar Gaussian estimates for $((\partial_t-\Delta_f^g)H_k)^{*j}$, which is
\begin{equation}
	\left|((\partial_t-\Delta_f^g)H_k)^{*j}\right|\le \frac{B^{j}\vol^g(M)^{j-1} t^{k-\frac{n}{2}+j-1}}{(k-\frac{n}{2}+1)\cdots(k-\frac{n}{2}+j-1)}e^{-\frac{d^g(x,y)^2}{4t}+A(x,y)}
\end{equation}

in particular we have $F_k\in C^0(M^n\times M^n\times [0,\infty))$. Moreover, similar arguments can be applied iteratively to show:
\begin{enumerate}
\item for any integer $\ell<k-\frac{n}{2}$, we have  $F_k\in C^\ell(M^n\times M^n\times [0,\infty))$  with
   \begin{align}\label{eq:esti F}
     \|F_k(\cdot,\cdot,t)\|_{L^{\infty}(M^n \times M^n)}&<C(f, g)t^{k-n/2}, \quad \forall t\in (0,1];
  \end{align}
\item  for any integer $k>\frac{n}{2}+2$, we have
\begin{equation}\label{eq:remainderEsti}
	\|H_k*F_k\|_{L^\infty({M^n\times M^n})}<C(f, g) t^{k+1-\frac{n}{2}},\qquad \forall t\in (0,1];
\end{equation}
and 
\begin{equation}\label{eq:remainderEstiGauss}
	\left\|{F_k*H_k}\cdot\exp\left(\frac{(\dist^g)^2}{4t}-A\right) \right\|_{L^\infty({M^n\times M^n})}\le C(f,g) t^{k+1-\frac{n}{2}},
\qquad \forall t\in (0,1].	
\end{equation}

\end{enumerate}
}
Next, given any multi-index $\alpha=(\alpha_1,\ldots, \alpha_n)$, for then we write,
\begin{align}
	\partial_x^{\alpha}(\partial_t-\Delta)H_k=t^{k-\frac{n}{2}-|\alpha|}\exp\left(-\frac{\dist^g(x,y)^2}{4t}+A(x,y)\right)G_{k,\alpha}(x,y,t). 
\end{align}
It follows from \eqref{eq:S_k esti} and a direct computation that $G_{k,\alpha}\in C^{\infty}(M^n\times M^n\times [0,1])$, now repeat the computation in \eqref{eq:convesti} replacing $G_k$ with $G_{k,\alpha}$, we get for any integer $l\in \setN$,

\begin{equation}\label{eq:C^l esti}
	\left\|{F_k*H_k} \right\|_{C^l({M^n\times M^n})}\le C(f,g,l) t^{k+1-\frac{n}{2}-l},
\qquad \forall t\in (0,1].	
\end{equation}

%{
%\color{red}Thanks to the compactness of $M^n$ (which in particular implies that the derivatives of $f$ are bounded), we know that
%for fixed $x,y$, $H_k(x,z,s), H_k(z,y,t-s)$ are both compactly supported, so
%$F_k$ is well-defined and that $F_k\in C^{\ell}(M^n\times M^n\times [0,\infty))$ for any integer $\ell <k-\frac{n}{2}$, see for instance an argument of \cite[p.152 Lemma 1]{Chavel}. It follows from a direct computation and induction that for any $t_0>0$,
 %  \begin{align}\label{eq:esti H}
 %      \left\|(\partial_t-\Delta_f^g)H_k(\cdot,\cdot,t)\right\|_{L^\infty(M^n\times M^n)}&<C(M^n,t_0)t^{k-n/2}, \quad \forall t\in (0,t_0].	\\
 %    \|F_k(\cdot,\cdot,t)\|_{L^{\infty}(M^n \times M^n)}&<C(M^n,t_0)t^{k-n/2}, \quad \forall t\in (0,t_0].\label{eq:esti F} 
 % \end{align}
%}

\begin{theorem}\label{thm:WeightedExpansion}
%{\color{red}Assume $|f|, |\nabla f|, |\mathrm{Hess}^g_f|\in C_b(M^n)$.}\footnote{As written above, let us discuss whether this assumption is needed or not.}
For {all $x, y \in M^n$ with $\dist^g(x, y) <\mathrm{inj}^g /4$,} the heat kernel $p_f(x,y,t)$ has the following asymptotic expansion:
  \begin{equation}\label{HKexpansion}
     p_f(x,y,t)=\frac{1}{(4\pi t)^{n/2}}\exp \left(-\frac{\dist^g(x,y)^2}{4t}+A(x, y)\right)\left(\sum_{j=0}^kt^ju_j(x,y)+O(t^{k+1})\right)
   \end{equation}
as $t \to 0^+$. Moreover if $x=y$, then the { remainder in the} expansion { has a} uniform { bound}; 
\begin{equation}\label{eq:UniConvP}
\sup_{x\in {M^n}, t<1} t^{\frac{n}{2}-k}\left| p_f(x,x,t)- \frac{1}{(4\pi t)^{n/2}}e^{f(x)}\sum_{j=0}^{k-1}t^ju_j(x,x)	 \right|	<\infty.
\end{equation}
 In particular, we have 
    \begin{equation}\label{u_1}
	    u_1(x,x)=\frac{\mathrm{Scal}^g(x)}{6}-\frac{1}{2}\Delta^g f(x)+\frac{1}{4}|\nabla f(x)|^2.
    \end{equation}
\end{theorem}

 \begin{proof}
    %It follows from the uniqueness of the heat kernel that $p_f=H_k-H_k*F_k$ is independent of $k$ when $k>n/2+2$. Moreover from (\ref{esti F}) we have
   % \begin{equation}
	    %\|t^{n/2-k-1}F_k*H_k(\cdot,\cdot,t)\|_{L^{\infty}(M^n\times M^n \times (0, 1])}<\infty.
   % \end{equation}
%Therefore the expansion up to order $k$ only depends on
 %  $H_k(x,y,t)=\alpha(x,y)S_k(x,y,t)$, when $\dist^g(x,y)<\frac{\epsilon}{2}$, $\alpha\equiv 1$, so (\ref{HKexpansion}) follows.
%{
It is shown above that $S_k$ hence $H_k$ has this expansion. { Note that $H_k-H_k*F_k$ also solves the heat equation}. From 
%that for every $k>\frac{n}{2}+2$, $H_k$ is a parametrix,
 (\ref{eq:esti H}), (\ref{eq:esti F}) { and the uniqueness of the heat kernel}, we infer that for every $k>\frac{n}{2}+2$, $p_f=H_k-H_k*F_k\in C^{k-\frac{n}{2}}(M^n\times M^n\times(0,\infty))$ (see also \cite[Thm 3.22]{Rosenberg}). 
 %{\color{red}hence $p_f\in C^\infty(M^n\times M^n\times (0,\infty))$.} {\color{purple}For all $x,y$ such that $\dist^g(x,y)\le\mathrm{inj}^g /4$,}
 Apply the inequality (\ref{eq:remainderEstiGauss}) to yield that  
 $$
 (p_f(x,y,t)-S_k(x,y,t))\cdot\exp\left(\frac{\dist^g(x,y)^2}{4t}-A(x,y)\right) =O(t^{k+1-n/2}),
 $$ 
 so $p_f$ has the same expansion as $S_k$ up to order $k-\frac{n}{2}$. %(see for instance \cite{Chavel}).
%}
%\footnote{I could not understand this part well. From the previous version it seems to me that $p_f=H_k-H_k *F_k$ for $k>n/2+2$. }
{In particular when $x=y$, %we have that
%\begin{equation}
%p_f(x,x,t)=\frac{1}{(4\pi t)^{n/2}}e^{f(x)}\sum_{j=0}^{\infty}t^ju_j(x,x)	
%\end{equation}
for each integer $k\ge1$, we have (\ref{eq:UniConvP}).}
 
    For the computation of $u_1$, recall in (\ref{u_jrecur}), we found that $u_0(x,y)=D^{-1/2}(y)$. Let $\gamma$ be as in Lemma \ref{solveU}, with (\ref{volexp}) we have
   \begin{equation}\label{u_0}
      u_0(x,y)= 1+\frac{1}{12}\mathrm{Ric}^g(\dot{\gamma}(0),\dot{\gamma}(0))\dist^g(x, y)^2+O(\dist^g(x, y)^3)	,
   \end{equation}
in particular $u_0(x,x)=1$. Then it follows that $\Delta^g u_0(x,x)=\mathrm{Scal}^g(x)/6$. Finally letting $y \to x$ in the second equation of (\ref{u_jrecur}) for $j=1$ leads to 
     \begin{equation*}
        u_1(x,x)=\Delta^g u_0(x,x)+\frac{1}{2}\Delta^g f(x)-\frac{1}{4}|\nabla f(x)|^2=\frac{\mathrm{Scal}^g(x)}{6}+\frac{1}{2}\Delta^g f(x)-\frac{1}{4}|\nabla f(x)|^2.
     \end{equation*}
 \end{proof}
 
%\begin{remark}
%As in \cite{CheegerYau}, the arguments above are also justified for the \textit{Neumann Laplacian}.
%\end{remark}

 \subsection{Divergence free property of the weighted Einstein tensor on a closed manifold}\label{subsec:divfree}
%From now on we make a further assumption that $M^n$ is a closed manifold. 
%{\color{red}then $|f|,|\nabla f|,|\mathrm{Hess}_f^g|\in C_b(M^n)$ is automatically true.}
 As discussed in section \ref{sec:2}, let us consider the heat kernel embedding: 
\begin{equation}
\Phi_{f, t}:M^n \hookrightarrow L^2(M^n, \mathrm{vol}_{f}^g)
\end{equation}
defined by
\begin{equation}
x \mapsto (y \mapsto p_f(x, y, t)).
\end{equation}
Put $g_{f, t}:=(\Phi_{f, t})^*g_{L^2}$.

%It follows from direct calculation that the corresponding Laplacian $\Delta_f^g$ viewed as an operator on a metric measure space discussed in Section \ref{sec:2} satisfies
%\begin{equation}\label{eq:witten}
%\Delta_f^g\phi=\Delta^g\phi-\langle \nabla f, \nabla \phi \rangle,\quad \forall \phi \in C^{\infty}(M^n),
%\end{equation}
%where $\Delta^g\phi=\mathrm{tr}(\mathrm{Hess}_{\phi})$. Note that the covariant derivatives $\nabla^g$ on $(M^n, \dist_g, \mathrm{vol}^g)$ and $\nabla^g_f$  on $(M^n, \dist^g, \mathrm{vol}_{f}^g)$ coincide from the point of view of $\RCD(K, N)$ spaces, which will be discussed in Section \ref{sec:4}, because of (\ref{eq:lipcheeger}), thus we denote both of them by $\nabla$. 

To study the second principal term of $g_{f, t}$ (recall (\ref{eq:firstpricipal}) for the first principal term in more general setting) along the same way as in \cite{BerardBessonGallot}, it is necessary to generalize the heat kernel expansion in \cite[p.380]{BerardBessonGallot} to weighted manifolds.  
We claim:

    \begin{theorem}[Weighted version of B\'erard-Besson-Gallot theorem]\label{thm:bbgweighted}
          We have the following asymptotic formula as $t \to 0^+$
      \begin{equation}
          c(n)t^{(n+2)/2}g_{f, t}=e^fg-e^f\left(\frac{2}{3}G^g-\di f \otimes\di f -\Delta^gfg+\frac{|\nabla f|^2}{2}g\right)t +O(t^2),
      \end{equation}
where {the remainder in the expansion has a uniform bound}; 
    \begin{equation}\label{eq:weightedBBGconv}
      \resizebox{.9\hsize}{!} {$\sup\limits_{ x\in M^n, t<1}\left| t^{-2}\left(c(n)t^{(n+2)/2}g_{f, t}- \left( e^fg-e^f\left(\frac{2}{3}G^g-\di f \otimes\di f -\Delta^gfg+\frac{|\nabla f|^2}{2}g\right)t\right) \right)\right|(x)<\infty.$}
    \end{equation}
In particular, we have the uniform convergence:
     \begin{equation}\label{eq:asymptoticweightedriemann}
         \left\|\frac{c(n)t^{(n+2)/2}g_{f, t}-e^fg}{t} - e^f\left(-\frac{2}{3}G^g+\di f \otimes\di f +\Delta^g fg-\frac{|\nabla f|^2}{2}g\right) \right\|_{L^{\infty}}   \to 0.
    \end{equation}
   \end{theorem}
   
\begin{proof}
By (\ref{DefPullBack}), which remains valid on weighted manifolds because of the characterization (\ref{eq:be}) for being an $\RCD(K, N)$ space, and the fact that the set of eigenfunctions $\{\phi_i\}_{i\geq0}$ forms an orthonormal basis of $L^2(M^n,\mathrm{vol}_f^g)$, we see that for every $x\in M^n$ and $v\in T_x M^n$,
    \begin{equation}\label{eq:ds}
    	         g_{f,t}(v,v)=\sum_{i}e^{-2\lambda_i t}|\di_x\phi_i(v)|^2=(\partial_y \partial_x p_f)_{(x,x,2t)}(v, v)=:(\di_Sp_f)_{(x, x, 2t)}(v, v)
    \end{equation}
where we used a fact that the expansion (\ref{eq:expansion1}) is satisfied in $C^{\infty}(M^n)$ {because of the elliptic estimates} (see {for instance} \cite[Thm.10.3]{Grig}), and we followed the notation in \cite{BerardBessonGallot}, denoting $\di_S:=\partial_y \partial_x$ for the mixed second derivative. {For the reader's convenience, let us clarify the meaning of this along \cite[p. 8]{Tewo}; for any smooth function $h:M^n\times M^n\to \setR$, and fixed $(x,y)$, we define maps $\di_1 h: T_x M^n\times M^n \to \setR$,  $\di_2 h: M^n\times T_y M^n \to \setR$ and  $\di_S h_{(x,y)}: T_x M^n\times T_y M^n\to \setR$ by $\di_1 h(v,y):= (\partial_x h(x,y))\cdot v$,   $\di_2 h(x,w)= (\partial_y h(x,y))\cdot w$ and  $\di_S h(v,w)= \partial_y(\partial_x h(x,y)\cdot v)\cdot w=\di_2(\di_1 h)$, respectively, for all $v\in T_x M^n, w\in T_y M^n$. %we call $\di_S h_{(x,x)}:T_xM^n\times T_xM^n\to \setR$ the mixed second. derivative of $h$ at $(x,x)$..
In order to} compute $(\di_S p_f)(x,x,2t)$, put 
{
\begin{equation}\label{eq:U}
U:=(4\pi t)^{n/2} \cdot \exp \left( \frac{\dist^g(x, y)^2}{4t}-A\right)p_f(x, y, t)=\sum_{j=0}^k t^j u_j(x,y)+I_{k+1}(x,y,t),
\end{equation}
where  $I_{k+1}(x, y, t)=O_{f, g}(t^{k+1})=O(t^{k+1})$.

}
 %from the regularity and uniform estimates (\ref{eq:esti H}), (\ref{eq:esti F}) of $H_k$ and $F_k$ respectively, 
%{\color{purple} Again, thanks to the compactness, the derivatives of $f$ are all bounded, by the same argument as in \cite[p.154]{Chavel}, we see that the differentiation of $U$ can be carried out term by term.} 

%Then {by \eqref{eq:C^l esti} and (\ref{eq:U}) (see also arguments in \cite[p.154]{Chavel})} we see 

Then, we show that for $x,y$ small enough, $\partial_xI_{k+1}=O(t^{k+1}), \partial_y I_{k+1}=O(t^{k+1})$ and $\di_SI_{k+1}=O(t^{k+1})$ hold, where ``$\partial_xI_{k+1}=O(t^{k+1})$'' means $|\partial_xI_{k+1}|=O(t^{k+1})$ (the same applies to $\partial_y I_{k+1}=O(t^{k+1})$), and ``$\di_SI_{k+1}=O(t^{k+1})$'' means $|\di_SI_{k+1}|=O(t^{k+1})$ with respect to the standard norm. %Here and in the sequel, when we use big-$O$ notation for $v\in T_xM^n$ we take pointwise norm $|v|_x$ and for $(v,w)\in T_xM^n\times T_y M^n$, we take the norm $|v|_x+|w|_y$.

To this end, fix $k$ and let $l=k+3$, note that $p_f=H_l-H_l*F_l$,  we see that for $x,y$ small enough such that $H_k=S_k$, it holds 
\begin{equation}\label{eq:mixdiffesti}
	I_{k+1}=U-\sum_{j=0}^k t^j u_j=\sum_{j=k+1}^l t^j u_j-(4\pi t)^{n/2} \cdot \exp \left( \frac{\dist^g(x, y)^2}{4t}-A\right)\cdot H_l*F_l,
\end{equation} 
it is clear that $\di_S$ of the first term on the RHS of \eqref{eq:mixdiffesti} is of $O(t^{k+1})$, for the second term, it follows from \eqref{eq:remainderEstiGauss} and
 \eqref{eq:C^l esti} that 
\begin{equation}
	\left\|\di_S \left(\exp \left( \frac{(\dist^g)^2}{4t}-A\right) H_l*F_l\right)\right\|_{L^{\infty}(M^n\times M^n)}\le C t^{k+1-\frac{n}{2}},\qquad \forall t\in(0,1].
\end{equation}
This completes the proof of $\di_SI_{k+1}=O(t^{k+1})$. The estimates $\partial_xI_{k+1}=O(t^{k+1})$, and $\partial_y I_{k+1}=O(t^{k+1})$ can be shown similarly.

Now we continue our computation of the expansion. It holds that
  \begin{equation*}
	 \resizebox{.9\hsize}{!} {$(8\pi t)^{n/2}(\di_S p_f)_{(x,y,2t)}=\left(-\frac{\di_Sr_x^2}{8t}e^AU-\frac{\partial_xr_x^2}{8t}\partial_y(e^AU)+\di_S(e^A U)\right)e^{-r_x^2/(8t)}-\frac{\partial_y r_x^2}{8t}\partial_xp_f$}
  \end{equation*}
where $r_x:=\dist^g(x, \cdot)$. Since at $(x,x)$, $\partial_xr_x^2=\partial_yr_x^2=0$ and $\di_S r_x^2=-2g$ hold in normal coordinates, we have 
    \begin{equation*}
	     (8\pi t)^{n/2}(\di_Sp_f)_{(x,x,2t)}=-\frac{e^{f(x)}U(x,x,2t)}{8t}(\di_Sr_x^2)_{(x,x, 2t)} +\di_S(e^AU)_{(x,x,2t)}
    \end{equation*}
Thanks to (\ref{u_0}) we have 
$(\partial_x u_0)_{(x,x, 2t)}=(\partial_y u_0)_{(x,x)}=0$ and $(\di_S u_0)_{(x,x)}=-\frac{1}{6}\mathrm{Ric}^g(x)$, which imply (recall the convention we use for big-$O$ notation of vectors)
$$
(\partial_x U)_{(x,x,2t)}=(\partial_x u_0)_{(x,x)}+O(t)=O(t).
$$ 
Similarly $(\partial_y U)_{(x,x,2t)}=O(t)$, and 
$$
(\di_S U)_{(x,x,2t)}=(\di_S u_0)_{(x,x)}+O(t)=-\frac{1}{6}\mathrm{Ric}^g(x)+O(t).
$$
It follows that
    \begin{align*}
	   \di_S(e^AU)_{(x,x,2t)}&=\left(U\di_S e^{A}+\partial_xe^{A}\partial_yU+\partial_ye^{A}\partial_x U+e^{A}\di_SU\right)_{(x,x,2t)}\\
	             &=\left(U\di_Se^{A}+e^{A}\di_SU+O(t)\right)_{(x,x,2t)}\\
	             &=\frac{1}{4}e^{f(x)} \di f\otimes \di f-\frac{1}{6}e^{f(x)}\mathrm{Ric}^g+O(t).
    \end{align*}
 This allows us to show that (recall $\di_S r_x^2=-2g$)
\begin{align*}
&(8\pi t)^{n/2}(\di_S p_f)_{(x,x,2t)}\\
&=\frac{1}{4t}e^{f(x)}\left(u_0(x,x)+2tu_1(x,x)+O(t^2)\right)g+\frac{1}{4}e^{f(x)} \di f\otimes \di f-\frac{1}{6}e^{f(x)}\mathrm{Ric}^g+O(t)
\end{align*}
%\begin{equation*}
%\resizebox{.9\hsize}{!} {$(8\pi t)^{n/2}(\di_S p_f)_{(x,x,2t)}=\frac{1}{4t}e^{f(x)}\left(u_0(x,x)+2tu_1(x,x)+O(t^2)\right)g+	\frac{1}{4}e^{f(x)} \di f\otimes \di f-\frac{1}{6}e^{f(x)}\mathrm{Ric}^g+O(t)$}
     %\end{equation*}
Recall that we have (\ref{u_1}), we finally deduce that 
   \begin{align*}
       4t(8\pi t)^{n/2}(\di_S p_f)_{(x,x,2t)}&=e^{f(x)}\left[1+2t\left(\frac{\mathrm{Scal}^g}{6}+\frac{\Delta^g f}{2}-\frac{|\nabla f|^2}{4}\right)\right]g\\
           &\quad +\frac{1}{2}e^{f(x)}\di f\otimes\di f\cdot 2t-\frac{1}{3}e^{f(x)}\mathrm{Ric}^g\cdot2t+O(t^2)\\
           &=e^{f}g-e^f\left(\frac{2}{3}G^g-\di f\otimes\di f -\Delta^gf g+\frac{|\nabla f|^2}{2}g\right)t +O(t^2)
    \end{align*}
 as claimed.
 \end{proof}

 Based on Theorem \ref{thm:bbgweighted}, let us give the following definitions in order to prove Corollary \ref{cor:weighted}.
 
 \begin{definition}[Weighted Einstein tensor]\label{def:einstein}
     Define {the weighted Einstein tensor} $G_f^g$ { for a closed weighted manifold $(M^n,\dist^g, \vol^g_f)$} by 
    \begin{equation}\label{eq:einsteinweight}
      G_f^g:=e^fG^g-\frac{3e^f}{2}\left(\di f\otimes\di f +\Delta^g fg -\frac{|\nabla f|^2}{2}g\right).
    \end{equation} 
 \end{definition}
 
 \begin{definition}[Weighted adjoint operator $\nabla_f^*$]\label{def:weightedadj}
  For any $T \in C^{\infty}((T^*)^{\otimes 2}M^n)$, define $\nabla^*_fT$ by
    \begin{equation}
       \nabla^*_fT:=\nabla^*T +T(\nabla f, \cdot ),
    \end{equation}
  where $\nabla^*$ is the adjoint operator of the covariant derivative $\nabla$ of $(M^n, g)$, namely $\nabla^*$ coincides with minus the divergence. Moreover we say that $T$ is \textit{divergence free on $(M^n, \dist^g, \mathrm{vol}^g_f)$} if $\nabla_f^*T=0$ holds.
\end{definition}

Note that $\nabla^*_f T$ is characterized by satisfying
    \begin{equation}
       \int_{M^n}\langle \nabla^*_fT, \omega \rangle \di \mathrm{vol}_{f}^g=\int_{M^n}\langle T, \nabla \omega \rangle \di \mathrm{vol}_{f}^g,\quad \forall \omega \in C^{\infty}(T^*M^n),
    \end{equation}
that is, $\nabla^*_f$ is the adjoint operator of the covariant derivative with respect to $\mathrm{vol}_{f}^g$.
Although the next proposition is a direct consequence of Theorem \ref{thm:bbgweighted} with more general results (Theorem \ref{thm:main} and Proposition \ref{prop:equivalence}), we give a direct proof.

\begin{proposition}\label{prop:directproof}
It holds that the weighted Einstein tensor $G_f^g$ is divergence free on $(M^n, \dist^g, \mathrm{vol}^g_f)$ if and only if $f$ is constant.
\end{proposition}

  \begin{proof}
    It is enough to check the ``only if'' part because the other implication reduces to (\ref{eq:divzero}).
    Assume that $\nabla_f^*G_f^g \equiv 0$ holds. Then it is easy to see
      \begin{equation}
         \nabla^*\left(\di f\otimes\di f +\Delta^g fg -\frac{|\nabla f|^2}{2}g\right)\equiv 0
      \end{equation}
    because of (\ref{eq:divzero}).
    Thus we have
       \begin{equation}\label{eq:tobeconstant}
         \Delta^g f\di f +\di \Delta^g f\equiv 0
       \end{equation}
     see also (\ref{eq:adjoint}).
     Let us consider an open subset $U$ of $M^n$:
       \begin{equation}
          U:=\{x \in M^n; \Delta^gf(x) \neq 0\}.
       \end{equation}
It is enough to prove $U = \emptyset$ because then $f$ is harmonic on $(M^n, g)$, thus $f$ is constant.
Assume $U \neq \emptyset$ and take $x \in U$. Define a function $F(z):=e^{f(z)}\Delta^g f(z)$.
Note that $F$ is locally constant on $U$ because %which comes from the equality on $U$:
\begin{equation}
\di F(z)=e^{f(z)}\Delta^gf(z)\di f(z)+e^{f(z)}\di \Delta^gf(z)=-e^{f(z)}\di \Delta^gf(z)+e^{f(z)}\di \Delta^gf(z)=0, %\di \left( f+ \log |\Delta^g f|\right) \equiv 0
\end{equation}
where we used (\ref{eq:tobeconstant}) in the second equality.
Let
\begin{equation}
X:=\{z \in M^n; F(z) = F(x)\} \subset U.
\end{equation}
Since $F$ is continuous on $M^n$, $X$ is closed in $M^n$. On the other hand since $F$ is locally constant on $U$, we see that $X$ is an open subset of $M^n$. Thus $X=M^n$. In particular 
    \begin{equation}
      0=\int_{M^n}\Delta^g f\di \mathrm{vol}^g=F(x)\int_Me^{-f}\di \mathrm{vol}^g \neq 0
    \end{equation}
which is a contradiction. Thus we have $U = \emptyset$.
\end{proof}

Finally, in connection with (\ref{eq:ahpt}), let us discuss the asymptotic behavior of 
\begin{equation}
t\mathrm{vol}_{f}^gB_{\sqrt{t}}(x)g_{f, t}.
\end{equation}

%From (\ref{eq:bbgrcd}), it is also natural to ask what can we say the asymptotic behavior of 
%\begin{equation}\label{eq:weightsecond}
%\frac{\frac{c(n)}{\omega_n}t\meas (B_{t^{1/2}}(x))g_t-g}{t}
%\end{equation}
%as $t \to 0^+$. More precisely, in connection with Theorem \ref{thm:main}, let us ask:
%\begin{itemize}
%\item When (\ref{eq:weightsecond}) is asymptotically divergence free as $t \to 0^+$? 
%\end{itemize}
%The following proposition tells us that it may be difficult to give an answer to this question even for general non-collapsed $\RCD(K, N)$ spaces because we do not know a suitable definition of having \textit{constant scalar curvature} for such spaces now.

\begin{proposition}\label{prop:weightedasymptotics}
We have the following uniform asymptotic expansion as $t \to 0^+$: 
\begin{equation}\label{eq:weightedasymptotic}
\frac{c(n)t}{\omega_n}\mathrm{vol}_{f}^g(B_{\sqrt{t}}(x))g_{f, t}=g-\frac{2t}{3}\left(G_f^g+ \frac{\mathrm{Scal}^g+3\Delta^g f-3|\nabla f|^2}{6(n+2)}g\right) +O(t^2).
\end{equation}
as $t \to 0^+$. In particular if $f$ is constant, then $\mathrm{Scal}^g$ is constant if and only if the second principal term of (\ref{eq:weightedasymptotic}) is divergence free on $(M^n, \dist^g, \mathrm{vol}^g)$, i.e.,
\begin{equation}\label{eq:divzero2}
\nabla^*\left(G^g+\frac{\mathrm{Scal}^g}{6(n+2)}g\right) \equiv 0.
\end{equation} 
\end{proposition}

\begin{proof}
The desired uniform convergence (\ref{eq:weightedasymptotic}) comes from (\ref{eq:asymptoticweightedriemann}) with Lemma \ref{lem:volasym}.
%because
%\begin{align}
%&te^{-f(x)}\mathrm{vol}_gB_{\sqrt{t}}(x)g_t\nonumber \\
%&=te^{-f(x)}\left( \omega_nt^{n/2} -\frac{\omega_n(\mathrm{Scal}^g+\Delta^g f-3|\nabla f|^2)}{6(n+2)} t^{(n+2)/2} +O(t^{(n+3)/2})\right)g_{f, t} \nonumber\\
%&=\frac{\omega_n}{c(n)}\left(g-\frac{t}{3}G +O(t^2)\right) -\frac{\omega_n\mathrm{Scal}_{M^n}^g}{6c(n)(n+2)}t \left(g-\frac{t}{3}G +O(t^2)\right) + O(t^{3/2})\nonumber\\
%&=\frac{\omega_n}{c(n)}g-\frac{\omega_nt}{3c(n)}\left(G-\frac{\mathrm{Scal}_{M^n}^g}{2(n+2)}\right) +O(t^{3/2})
%\end{align}
%which implies (\ref{eq:weightedasymptotic}).
For the remaining statement, we assume that $f$ is constant. Then thanks to (\ref{eq:divzero}), we have
\begin{equation}
\nabla^*\left( G^g+ \frac{\mathrm{Scal}^g}{6(n+2)}g\right)=0 \Longleftrightarrow \nabla^*(\mathrm{Scal}^gg)=0 \Longleftrightarrow \di \mathrm{Scal}^g=0
\end{equation}
which proves the desired equivalence, where we used $\nabla^*g=0$.
%Note that the (LHS) of (\ref{eq:weightedasymptotic}) is asymptotically divergence free as $t \to 0^+$ if and only if the (RHS) of (\ref{eq:weightedasymptotic}) is divergence free because the convergence of (\ref{eq:weightedasymptotic}) is uniform.

%If $\mathrm{Scal}_{M^n}^g$ is a constant function, then it is trivial that the (RHS) of (\ref{eq:weightedasymptotic}) is divergence free because of (\ref{eq:divzero}).
%Assume that the (RHS) of (\ref{eq:weightedasymptotic}) is divergence free. Then
%\begin{equation}\label{eq:scaldivzero}
%\int_{M^n}\mathrm{Scal}_{M^n}^g\mathrm{tr}\nabla \omega \di \mathrm{vol}_g=0
%\end{equation}

%for any smooth $1$-form $\omega$ on $M^n$.
%In particular for any eigenfunction $f$ of $\Delta$ whose eigenvalue is not zero, applying (\ref{eq:scaldivzero}) for $\omega =\dist f$ yields
%\begin{equation}
%\int_{M^n}\mathrm{Scal}_{M^n}^gf\di \mathrm{vol}_g=0
%\end{equation}
%which shows that $\mathrm{Scal}_{M^n}^g$ is a constant function because $f$ is arbitrary. Thus we conclude the proof.
\end{proof}
It is an immediate consequence of Proposition \ref{prop:weightedasymptotics} that for a given compact non-collapsed $\RCD(K, N)$ space $(X, \dist, \mathcal{H}^N)$, it is hard to check directly the weakly asymptotically divergence free property of the second principal term of $t\mathcal{H}^n(B_{\sqrt{t}}(x))g_t$. 

%it is hard now to establish a weakly asymptotically divergence free property of the second principal term of $t\mathcal{H}^n(B_{\sqrt{t}}(x))g_t$ for a ``nice class'' of compact non-collapsed $\RCD(K, n)$ spaces $(X, \dist, \mathcal{H}^n)$.

\section{Second principal term in $\RCD$ case; proof of Theorem \ref{thm:main}}\label{sec:4}
The main purpose of this section is to prove Theorem \ref{thm:main}. For that let us fix the terminology borrowed from \cite{Gigli} minimally.
\subsection{Second order differential calculus; list of differential operators}
Throughout this subsection we fix an $\RCD(K, \infty)$ space $(X, \dist, \meas)$. The space of all test functions due to \cite{Gigli, Savare} is defined by
\begin{equation}
\mathrm{Test}F(X, \dist, \meas):=\left\{ f \in \mathrm{Lip}_b(X, \dist) \cap D(\Delta); \Delta f \in H^{1, 2}(X, \dist, \meas)\right\}
\end{equation}
which is an algebra.
We first recall the Hessian for a test function (see also \cite[Def.3.3.1]{Gigli}).
\begin{theorem}[Hessian]
For any $f \in \mathrm{Test}F(X, \dist, \meas)$, there exists $T \in L^2((T^*)^{\otimes 2}(X, \dist, \meas))$ such that for any $f_i \in \mathrm{Test}F(X, \dist, \meas)$, $i=1,2$,
\begin{equation}\label{eq:hess}
T(\nabla f_1, \nabla f_2)=\frac{1}{2}\left(\langle \nabla f_1, \nabla \langle \nabla f_2, \nabla f\rangle \rangle + \langle \nabla f_2, \nabla \langle \nabla f_1, \nabla f \rangle \rangle -\langle f, \nabla \langle \nabla f_1, \nabla f_2\rangle \rangle \right)
\end{equation}
holds for $\meas$-a.e. $x \in X$. Since $T$ is unique, we denote it by $\mathrm{Hess}_f$ and call it the \textit{Hessian} of $f$.
\end{theorem}
See \cite[Thm.3.3.2 and 3.3.8]{Gigli}. {For the reader's convenience}, let us provide a proof of the uniqueness. First let us recall that the space of all test tensor fields of type $(0, 2)$;
\begin{equation}\label{eq:testtensor}
\mathrm{Test}(T^*)^{\otimes 2}(X, \dist, \meas):=\left\{ \sum_{i=1}^lf_{i, 0}\di f_{i, 1} \otimes \di f_{i, 2}; l \in \mathbb{N}, f_{i, j} \in \mathrm{Test}F(X, \dist, \meas)\right\}
\end{equation}
is dense in $L^2((T^*)^{\otimes 2}(X, \dist, \meas))$ (see \cite[(3.2.7)]{Gigli}). For all $T_1, T_2 \in L^2((T^*)^{\otimes 2}(X, \dist, \meas))$ satisfying (\ref{eq:hess}) as $T=T_i$, we have 
\begin{equation}\label{eq:uniq}
\int_X\langle T_1-T_2, S\rangle \di \meas =0,\quad \forall S \in \mathrm{Test}(T^*)^{\otimes 2}(X, \dist, \meas).
\end{equation}
The density of (\ref{eq:testtensor}) in $L^2((T^*)^{\otimes 2}(X, \dist, \meas))$ allows us to take $S$ as $T_1-T_2$ in (\ref{eq:uniq}), thus we conclude the proof. Note that the uniquness appeared below can be checked similarly by the density of test objects in $L^2$. 

Moreover it is proved in \cite[Cor.3.3.9]{Gigli} that the Hessian is well-defined for any $f \in D(\Delta)$ satisfying (\ref{eq:hess}) and the Bochner inequality involving the Hessian term :
\begin{equation}\label{eq:bochner}
\frac{1}{2}\int_X|\nabla f|^2\Delta \phi \di \meas \ge \int_X\phi\left( |\mathrm{Hess}_f|^2+\langle \nabla \Delta f, \nabla f\rangle +K|\nabla f|^2\right)\di \meas
\end{equation}
holds for all $f, \phi \in D(\Delta)$ with $\phi \ge 0$ and $\phi, \Delta \phi \in L^{\infty}(X, \meas)$. In particular we have
\begin{equation}\label{eq:hessbound}
\int_X|\mathrm{Hess}_f|^2\di \meas \le \int_X\left((\Delta f)^2-K|\nabla f|^2\right)\di \meas, \quad \forall f \in D(\Delta).
\end{equation}
\begin{definition}[Adjoint operator $\delta$]\label{def:adj} 
 Let us denote by $D(\delta)$ the set of $\omega \in L^2(T^*(X, \dist, \meas))$ such that there exists $f \in L^2(X, \meas)$ such that 
\begin{equation}
\int_X\langle \omega, \di h \rangle \di \meas= \int_Xfh\di \meas, \quad \forall h \in H^{1, 2}(X, \dist, \meas)
\end{equation}
holds. Since $f$ is unique, we denote it by $\delta\omega$.
\end{definition}
See also \cite[Def.3.5.11]{Gigli}.
Let us define the space of test $1$-forms:
\begin{equation}
\mathrm{Test}T^*(X, \dist, \meas):=\left\{\sum_{i=1}^lf_{0, i}\di f_{1, i}; l\in \mathbb{N}, f_{j, i} \in \mathrm{Test}F(X, \dist, \meas)\right\}.
\end{equation}
It is proved in \cite[Prop.3.5.12]{Gigli} that $\mathrm{Test}T^*(X, \dist, \meas) \subset D(\delta)$ holds with \begin{equation}\label{eq;delta}
\delta(f_1\di f_2)=-\langle \nabla f_1, \nabla f_2 \rangle-f_1\Delta f_2, \quad \forall f_i \in \mathrm{Test}F(X, \dist, \meas).
\end{equation}
\begin{definition}[Sobolev space $W^{1, 2}_C$]\label{def:cov}
Let us denote by $W^{1, 2}_C(T^*(X, \dist, \meas))$ the set of all $\omega \in L^2(T^*(X, \dist, \meas))$ such that there exists $T \in L^2((T^*)^{\otimes 2}(X, \dist, \meas))$ such that 
\begin{equation}
\int_X\langle T, f_0\di f_1\otimes \di f_2\rangle \di \meas = \int_X\left( -\langle \omega, \di f_2\rangle \delta(f_0\di f_1)-f_0\langle \mathrm{Hess}_{f_2}, \omega \otimes \di f_1\rangle \right) \di \meas
\end{equation}
holds. Since $T$ is unique, we denote it by $\nabla \omega$.
\end{definition}
See also \cite[Def.3.4.1]{Gigli}.
Comparing our working definition above for $W^{1,2}_C$-$1$-forms with Gigli's one for $W^{1,2}_C$-vector fields \cite[Def.3.4.1]{Gigli}, it is easy to see that for any $\omega \in L^2(T^*(X, \dist, \meas))$, $\omega \in W^{1, 2}_C(T^*(X, \dist, \meas)$ holds if and only if $\omega^{\sharp} \in W^{1, 2}_C(T(X, \dist, \meas))$ holds, where we used the canoncal musical isomorphism $L^2(T^*(X, \dist, \meas)) \simeq L^2(T(X, \dist, \meas))$.
It is proved in \cite[Thm.3.4.2]{Gigli} that $\mathrm{Test}T^*(X, \dist, \meas) \subset W^{1, 2}_C(T^*(X, \dist, \meas))$ holds with
\begin{equation}\label{eq:leib}
\nabla (f_1\di f_2)=\di f_1\otimes \di f_2 +f_1 \mathrm{Hess}_{f_2}, \quad \forall f_i \in \mathrm{Test}F(X, \dist, \meas).
\end{equation}
\begin{definition}[Sobolev space $H^{1, 2}_C$]\label{def:cov2}
Let us denote by $H^{1, 2}_C(T^*(X, \dist, \meas))$ the closure of $\mathrm{Test}T^*(X, \dist, \meas)$ in $W^{1, 2}_C(T^*(X, \dist, \meas))$.
\end{definition}
See also \cite[Def.3.4.3]{Gigli}.
\begin{definition}[Exterior derivative $\di$]\label{def:ext}
Let us denote by $W^{1, 2}_{\di}(T^*(X, \dist, \meas))$ the set of all $\omega \in L^2(T^*(X, \dist, \meas))$ such that there exists $\eta \in L^2(\bigwedge^2T^*(X, \dist, \meas))$ such that 
\begin{equation}
\int_X\langle \eta, \alpha_0\otimes \alpha_1\rangle\di \meas=\int_X\left(\langle \omega, \alpha_0\rangle \delta\alpha_1-\langle \omega, \alpha_1\rangle \delta \alpha_0\right)\di \meas, \quad \forall \alpha \in \mathrm{Test}T^*(X, \dist, \meas)
\end{equation}
holds. Since $\eta$ is unique, we denote it by $\di \omega$. 
\end{definition}
See also \cite[Def.3.5.1]{Gigli}.
It is proved in \cite[Thm.3.5.2]{Gigli} that $\mathrm{Test}T^*(X, \dist, \meas) \subset W^{1, 2}_{\di}(T^*(X, \dist, \meas))$ holds.
\begin{definition}[Sobolev space $H^{1, 2}_H$]\label{def:sob2}
Let us denote by $H^{1, 2}_H(T^*(X, \dist, \meas))$ the completion of $\mathrm{Test}T^*(X, \dist, \meas)$ with respect to the norm:
\begin{equation}
\|\omega\|_{H^{1, 2}_H}^2:=\|\omega\|_{L^2}^2+\|\delta \omega\|_{L^2}^2+\|\di \omega\|_{L^2}^2.
\end{equation}
\end{definition}
See also \cite[Def.3.5.13]{Gigli}
\begin{definition}[Hodge Laplacian $\Delta_{H, 1}$]\label{def:hodge}
Let us denote by $D(\Delta_{H, 1})$ the set of all $\omega \in H^{1, 2}_H(T^*(X, \dist, \meas))$ such that there exists $\eta \in L^2(T^*(X, \dist, \meas))$ such that 
\begin{equation}
\int_X\left( \langle \di \omega, \di \alpha \rangle +\delta \omega \cdot \delta \alpha \right)\di \meas=\int_X\langle \eta, \alpha \rangle \di \meas, \quad \forall \alpha \in H^{1, 2}_H(T^*(X, \dist, \meas))
\end{equation}
holds. Since $\eta$ is unique, we denote it by $\Delta_{H, 1}\omega$.
\end{definition}
See also \cite[Def.3.5.14]{Gigli}.
It is proved in \cite[Cor.3.6.4]{Gigli} that $H^{1, 2}_H(T^*(X, \dist, \meas)) \subset H^{1, 2}_C(T^*(X, \dist, \meas))$ holds with 
\begin{equation}\label{eq:boch}
\int_X|\nabla \omega|^2\di \meas \le \int_X( |\di \omega|^2+|\delta \omega|^2-K|\omega|^2) \di \meas, \quad \forall \omega \in H^{1, 2}_H(T^*(X, \dist, \meas)).
\end{equation}
On the other hand it follows from Definitions \ref{def:cov} and \ref{def:ext} that for any $\omega \in H^{1, 2}_C(T^*(X, \dist, \meas))$, 
\begin{equation}\label{eq:ext}
\di \omega (V_1, V_2)=(\nabla_{V_1}\omega) (V_2) -(\nabla_{V_2}\omega) (V_1), \quad \forall V_i \in L^{\infty}(T(X, \dist,  \meas))
\end{equation}
holds, where $\nabla_{V_1}\omega:=\nabla \omega ( \cdot, V_1)$. In particular, we see that $H^{1, 2}_C(T^*(X, \dist, \meas))$ is a subset of $H^{1, 2}_{\di}(T^*(X, \dist, \meas))$, where $H^{1, 2}_{\di}(T^*(X, \dist, \meas))$ denotes the $W^{1, 2}_{\di}$-closure of $\mathrm{Test}T^*(X, \dist, \meas)$, with
\begin{equation}\label{eq:point}
|\di \omega|^2 \le 2|\nabla \omega|^2, \quad \meas-a.e. \,x \in X
\end{equation}
for any $\omega \in H^{1, 2}_C(T^*(X, \dist, \meas))$.
\begin{definition}[Adjoint operator $\nabla^*$]
Let us denote by $D(\nabla^*)$ the set of all $T \in L^2((T^*)^{\otimes 2}(X, \dist, \meas))$ such that there exists $\eta \in L^2(T^*(X, \dist, \meas))$ such that 
\begin{equation}
\int_X\langle T, \nabla \omega \rangle \di \meas =-\int_X\langle \eta, \omega \rangle \di \meas, \quad \forall \omega \in H^{1, 2}_C(T^*(X, \dist, \meas))
\end{equation}
holds. Since $\eta$ is unique, we denote it by $\nabla^*T$. We say $T$ is \textit{divergence free} if $\nabla^*T=0$ holds.
\end{definition}
See also \cite[Def. 2.17]{Honda20}.
Note that for any $f \in \mathrm{Test}F(X, \dist, \meas)$ we have $\di f \otimes \di f \in D(\nabla^*)$ with
\begin{equation}\label{eq:adjoint}
\nabla^*(\di f \otimes \di f)=-\Delta f\di f-\frac{1}{2}\di |\nabla f|^2.
\end{equation} 
See for instance \cite[Prop.2.18]{Honda20} for the proof. 
Finally let us recall the following result proved in \cite[Prop.3.2]{Han} in the finite dimensional (maximal) case. Note that for any tensor $T$ of type $(0, 2)$ on $X$, the trace $\mathrm{tr}(T)$ is defined by $\mathrm{tr}(T):=\langle T, g\rangle$.
\begin{theorem}[Laplacian is trace of Hessian under maximal dimension]\label{thm:hanresult}
Assume that $N$ is an integer with $\mathrm{dim}_{\dist, \meas}(X)=N$. Then
for all $f \in D(\Delta)$ we see that 
\begin{equation}\label{eq:laptrace}
\Delta f=\mathrm{tr}(\mathrm{Hess}_f) \quad \text{for $\meas$-a.e. $x \in X$}.
\end{equation}
\end{theorem}
Compare with (\ref{eq:witten}). We can also reprove (\ref{eq:laptrace}) along the main tools in the paper when $(X, \dist)$ is compact, see (\ref{eq:alternativeproof}).
\subsection{A key formula}
Throughout this subsection let us fix a compact $\RCD(K, N)$ space $(X, \dist, \meas)$.
\begin{theorem}[Laplacian of $(X, g_t, \meas)$]\label{thm:pullbacklap}
For any $f \in D(\Delta)$ and any $\phi \in H^{1, 2}(X, \dist, \meas) \cap L^{\infty}(X, \meas)$, we have
\begin{equation}\label{eq:lapt}
\int_Xg_t(\nabla f, \nabla \phi)\di \meas=-\int_X\phi\Delta^tf\di \meas,
\end{equation}
where 
\begin{equation}
\Delta^t f=\langle g_t, \mathrm{Hess}_f\rangle +\frac{1}{4}\langle \nabla f, \nabla_x \Delta_x p(x, x, 2t)\rangle \in L^1(X, \meas).
\end{equation}
\end{theorem}
See \cite[Thm.3.4]{Honda19} for the proof.
Let us give a remark on Theorem \ref{thm:pullbacklap} that if $\meas = \mathcal{H}^N$ (that is, $(X, \dist, \mathcal{H}^N)$ is a non-collapsed $\RCD(K, N)$ space), then multiplying by $t^{(N+2)/2}$ on both sides of (\ref{eq:lapt}) and then letting $t \to 0^+$ show
\begin{equation}\label{eq:alternativeproof}
\int_X\langle \nabla f, \nabla \phi\rangle \di \mathcal{H}^N=-\int_X\mathrm{tr}(\mathrm{Hess}_f)\phi\di \mathcal{H}^N.
\end{equation}
Since $H^{1, 2}(X, \dist, \meas) \cap L^{\infty}(X, \meas)$ is dense in $H^{1, 2}(X, \dist, \meas)$, (\ref{eq:alternativeproof}) is also satisfied for any $\phi \in H^{1, 2}(X, \dist, \meas)$. Thus by definition of $D(\Delta)$ we have (\ref{eq:laptrace}). In particular since \cite[Cor.1.3]{Honda19} proves that $\mathrm{dim}_{\dist, \meas}(X)=N$ implies $\meas=c\mathcal{H}^N$ for some $c \in (0, \infty)$, we reprove Theorem \ref{thm:hanresult}.
%which proves (\ref{eq:laptrace}) because $H^{1, 2}(X, \dist, \mathcal{H}^N) \cap L^{\infty}(X, \meas)$ is dense in $H^{1, 2}(X, \dist, \meas)$. This (with \cite[Cor.1.3]{Honda19}) gives an alternative proof of Theorem \ref{thm:hanresult} in this setting.

\begin{proposition}\label{prop:divfree}
For any $\omega \in H^{1, 2}_C(T^*(X, \dist, \meas))$ and any $t \in (0, \infty)$ we have
\begin{equation}\label{eq:generalformula}
\int_X\langle g_t, \nabla \omega \rangle \di \meas =-\frac{1}{4}\int_X\langle \omega,  \di_x \Delta_xp(x, x, 2t)\rangle \di \meas.
\end{equation}
\end{proposition}
\begin{proof}
It follows from (\ref{eq:lapt}) and (\ref{eq:leib}) that if $\omega=f_1\di f_2$ holds for some $f_i \in \mathrm{Test}F(X, \dist, \meas)$, then we have
\begin{align}
\int_X\langle g_t, \nabla \omega \rangle \di \meas &=\int_X\langle g_t, \di f_1 \otimes \di f_2+ f_1\mathrm{Hess}_{f_2} \rangle \di \meas \nonumber \\
&=-\frac{1}{4}\int_X\langle f_1\di f_2, \di_x \Delta_xp(x, x, 2t)\rangle \di \meas=-\frac{1}{4}\int_X\langle \omega,  \di_x \Delta_xp(x, x, 2t)\rangle \di \meas,
\end{align}
which easily implies the conclusion because by definition $\mathrm{Test}T^*(X, \dist, \meas)$ is dense in $H^{1, 2}_C(T^*(X, \dist, \meas))$.
\end{proof}
It is proved in \cite[Prop.3.6.1]{Gigli} that for all $f \in \mathrm{Test}F(X, \dist, \meas)$ we have $\di f \in D(\Delta_{H, 1})$ with
\begin{equation}\label{eq:commu}
\Delta_{H, 1} (\di f)=-\di \Delta f.
\end{equation}
\begin{lemma}\label{lem:commu}
For fixed $t \in (0, \infty)$, the function $x \mapsto p(x, x, t)$ is in $\mathrm{Test}F(X, \dist, \meas)$.. In particular we have $\di_xp(x, x, t) \in D(\Delta_{H, 1})$ with $\Delta_{H, 1}(\di_xp(x, x, t))=-\di_x \Delta_xp(x, x, t)$.
\end{lemma}
\begin{proof}
Since for fixed $l \in \mathbb{N}$, (\ref{eq:eigenest}) and (\ref{eq:hessbound}) show
\begin{align}
&\left|\sum_i^le^{-\lambda_i t} \phi_i^2\right|\le (C_5)^2\sum_ie^{-\lambda_it}\lambda_i^{N/2}<\infty, \\
&\left|\nabla \left(\sum_i^le^{-\lambda_it}\phi_i^2\right)\right|\le 2(C_5)^2\sum_ie^{-\lambda_it}\lambda_i^{(N+1)/2}<\infty, \\
&\resizebox{.9\hsize}{!}{$\left|\Delta \left(\sum\limits_i^le^{-\lambda t} \phi_i^2\right)\right|\le 2 \sum\limits_i^le^{-\lambda_it}\left(|\nabla \phi_i|^2 +|\phi_i| |\Delta \phi_i| \right) \le 4(C_5)^2\sum_ie^{-\lambda_it}\lambda_i^{(N+2)/2}<\infty$}
\end{align}
and 
\begin{align}
	 &\int_X\left|\nabla \left(\Delta \left(\sum_i^le^{-\lambda t} \phi_i^2\right)\right) \right|^2 \di \meas \nonumber \\
     &=2\int_X\left|\nabla \left(\sum_i^le^{-\lambda_it}\left(|\nabla \phi_i|^2 +\lambda_i\phi_i^2\right)\right)\right|^2\di \meas \nonumber \\
     &=2\,\resizebox{.9\hsize}{!}{$\sum\limits_i^le^{-2\lambda_it}\int_X\left(|\nabla |\nabla \phi_i|^2|^2 +4\lambda_i^2\phi_i^2|\nabla \phi_i|^2\right)\di \meas + 2\sum\limits_{i\neq j}^le^{-(\lambda_i+\lambda_j)t}\lambda_j\int_X\phi_j\mathrm{Hess}_{\phi_i}(\nabla \phi_i, \nabla \phi_j)\di \meas \nonumber $}\\
     &\le C(K, N, t)<\infty
\end{align}
letting $l \to \infty$ in above inequalities with Mazur's lemma completes the proof. 
{For the reader's convenience}, let us provide a proof as follows.

First it is easy to see that $D(\Delta)$ is a Hilbert space equipped with the norm 
$
\|f\|_{D}=\left(\|f\|_{H^{1, 2}}^2+\|\Delta f\|_{L^2}^2\right)^{1/2}.
$
Since the estimates above show that the sequence $\{\sum_{i}^le^{-\lambda_i t} \phi_i^2\}_l$ is bounded in $D(\Delta)$, we have a weak convergent subsequence to some $f$ in $D(\Delta)$. Then applying Mazur's lemma yields that this is a strong convergence because the sequence consists of linear combinations of  $e^{-\lambda_i t} \phi_i^2$. Since $\sum_{i}^le^{-\lambda_i t} \phi_i^2(x) \to p(x, x, t)$ in $L^2(X, \meas)$ as $l \to \infty$, we have $f(x)=p(x, x, t) \in D(\Delta)$. Moreover the estimates above also imply the equi-Lipschitz continuity of  $\{\sum_{i}^le^{-\lambda_i t} \phi_i^2\}_l$. Thus $f(x)=p(x, x, t)$ is Lipschitz. Similarly, applying Mazur's lemma for a sequence $\{\Delta \sum_{i}^le^{-\lambda_i t} \phi_i^2\}_l$ in $H^{1, 2}(X, \dist, \meas)$ yields that $\Delta p(x, x, t) \in H^{1, 2}(X, \dist, \meas)$. Thus $p(x, x, t) \in \mathrm{Test}F(X, \dist, \meas)$.
The remaining statement comes from (\ref{eq:commu}) as $f(x)=p(x, x, t)$.
\end{proof}
We are now in position to prove a technical key result which will play a role in the proof of Theorem \ref{thm:main}.
\begin{theorem}\label{thm:technical}
For any $t \in (0, \infty)$ and any $\omega \in D(\Delta_{H, 1})$ with $\Delta_{H, 1}\omega \in D(\delta)$ we have
\begin{equation}\label{eq:test}
\int_X\langle g_t, \nabla \omega\rangle \di \meas =\frac{1}{4}\int_X\delta (\Delta_{H, 1}\omega) p(x, x, 2t)\di \meas. 
\end{equation}
\end{theorem}
\begin{proof}
Proposition \ref{prop:divfree} and Lemma \ref{lem:commu} yield
\begin{align*}
\int_X\langle g_t, \nabla \omega\rangle \di \meas 
&=-\frac{1}{4}\int_X\langle \omega, \di_x \Delta_xp(x, x, 2t)\rangle \di \meas =\frac{1}{4}\int_X\langle \omega, \Delta_{H, 1}(\di_xp(x, x, 2t))\rangle\di \meas \\
&=\frac{1}{4}\int_X\langle \Delta_{H, 1}\omega, \di_x p(x, x, 2t)\rangle \di \meas = \frac{1}{4}\int_X\delta (\Delta_{H, 1}\omega) p(x, x, 2t)\di \meas.
\end{align*}
\end{proof}
\subsection{Non-collapsed $\RCD(K, N)$ space and fine properties on Sobolev spaces}
The main purpose of this subsection is to recall the definition of non-collapsed $\RCD(K, N)$ spaces and to introduce fine properties on Sobolev spaces of the spaces.
Non-collapsed $\RCD(K, N)$ spaces are introduced in \cite{DePhillippisGigli} as a synthetic counterpart of non-collapsed Ricci limit spaces. The definition is as follows.
\begin{definition}[Non-collapsed $\RCD(K, N)$ space]
An $\RCD(K, N)$ space $(X, \dist, \meas)$ is said to be \textit{non-collapsed} if $\meas=\mathcal{H}^N$ holds.
\end{definition}
Non-collapsed $\RCD(K, N)$ space have nicer properties than that of general $\RCD(K, N)$ spaces. 
For example we have
\begin{equation}\label{eq:coinci}
H^{1, 2}_H(T^*(X, \dist, \mathcal{H}^N))=H^{1, 2}_C(T^*(X, \dist, \mathcal{H}^N)),
\end{equation}
which is a direct consequence of the following result proved in \cite[Prop.4.1]{Han}, see Corollary \ref{cor:equiv}. 
\begin{theorem}\label{thm:divtr}
Let $(X, \dist, \mathcal{H}^N)$ be a non-collapsed $\RCD(K, N)$ space. Then
we have $H^{1, 2}_C(T^*(X, \dist, \mathcal{H}^N)) \subset D(\delta)$ with
\begin{equation}\label{eq:deltatr}
\delta \omega=-\mathrm{tr} \nabla \omega, \quad \forall \omega \in H^{1, 2}_C(T^*(X, \dist, \mathcal{H}^N)).
\end{equation}
\end{theorem}
\begin{proof}
Theorem \ref{thm:hanresult} with (\ref{eq;delta}) yields that for all $f_i \in \mathrm{Test}F(X, \dist, \mathcal{H}^N)$,
\begin{align*}
\delta(f_1\di f_2)&=-\langle \di f_1, \di f_2\rangle -f_1\Delta f_2 \\
&=-\langle \di f_1, \di f_2\rangle -f_1\mathrm{tr}(\mathrm{Hess}_{f_2}) \\
&=-\langle g, \di f_1 \otimes \di f_2\rangle -\langle g, f_1\mathrm{Hess}_{f_2}\rangle =-\langle g, \nabla (f_1\di f_2)\rangle =-\mathrm{tr}\nabla (f_1\di f_2)
\end{align*} 
holds, which shows that (\ref{eq:deltatr}) holds for all $\omega \in \mathrm{Test}T^*(X, \dist, \mathcal{H}^N)$.
Thus we have the conclusion because by definition $\mathrm{Test}T^*(X, \dist, \mathcal{H}^N)$ is dense in $H^{1, 2}_C(T^*(X, \dist, \mathcal{H}^N))$.
\end{proof}
It directly follows from Theorem \ref{thm:divtr} that for a non-collapsed $\RCD(K, N)$ space $(X, \dist, \mathcal{H}^N)$ and any $f \in D(\Delta)$, we have $fg \in D(\nabla^*)$ with 
\begin{equation}
\nabla^*(fg)=-\di f
\end{equation}
because for any $\omega \in H^{1, 2}_C(T^*(X, \dist, \mathcal{H}^N))$, 
\begin{equation}
\int_X\langle \omega, \nabla^*(fg)\rangle \di \mathcal{H}^N=\int_X\langle \nabla \omega, fg\rangle \di \mathcal{H}^N =\int_Xf\delta \omega \di \mathcal{H}^N=\int_X\langle \di f, \omega\rangle \di \mathcal{H}^N.
\end{equation}
The following is also a direct consequence of (\ref{eq:boch}), (\ref{eq:point}) and Theorem \ref{thm:divtr}:
\begin{corollary}\label{cor:equiv}
Let $(X, \dist, \mathcal{H}^N)$ be a non-collapsed $\RCD(K, N)$ space. Then we have $H^{1, 2}_H(T^*(X, \dist, \mathcal{H}^N))=H^{1, 2}_C(T^*(X, \dist, \mathcal{H}^N))$ with 
\begin{equation}
\frac{1}{2}\|\omega\|_{H^{1, 2}_H}\le \|\omega\|_{H^{1, 2}_C} \le (1+K^-)\|\omega\|_{H^{1, 2}_H}, \quad \forall \omega \in H^{1, 2}_H(T^*(X, \dist, \mathcal{H}^N)),
\end{equation}
where $K^-=\max \{0, -K\}$.
\end{corollary}
It is proved in \cite{DePhillippisGigli} that any non-collapsed $\RCD(K, N)$ space $(X, \dist, \mathcal{H}^N)$ satisfies $\mathrm{dim}_{\dist, \meas}(X)=N$. It is also conjectured that the converse implication is true up to multiplying a positive constant to the measure, that is, if a $\RCD(K, N)$ space $(X, \dist, \meas)$ satisfies $\mathrm{dim}_{\dist, \meas}(X)=N$, then $\meas=a\mathcal{H}^N$ holds for some $a \in (0 ,\infty)$. Note that by definition, the $\RCD(K, N)$ condition is unchanged after multiplying a positive constant to the measure; if $(X, \dist, \meas)$ is an $\RCD(K, N)$ space, then $(X, \dist, a\meas)$ is also an $\RCD(K, N)$ space for any $a \in (0, \infty)$. Therefore  $(X, \dist, a\mathcal{H}^N)$ is an $\RCD(K, N)$ space for some $a \in (0, \infty)$, then $(X, \dist, \mathcal{H}^N)$ is a non-collapsed $\RCD(K, N)$ space.
Thus the conjecture states that the maximality of the essential dimension characterizes the non-collapsed condition.

It is proved in \cite[Cor.1.3]{Honda19} that the conjecture is true when $(X, \dist)$ is compact. 
Finally we introduce another characterization for being a non-collapsed $\RCD(K, N)$ space proved in \cite[Cor.4.2]{Honda19}:
\begin{theorem}[Characterization of non-collapsed $\RCD(K, N)$ space]\label{thm:charnoncollapsed}
Let $(X, \dist, \mathcal{H}^n)$ be a compact $\RCD(K, N)$ space. Then the following two conditions are equivalent:
\begin{enumerate}
\item $(X, \dist, \mathcal{H}^n)$ is a non-collapsed $\RCD(K, n)$ space.
\item We have
\begin{equation}
\inf_{x \in X, r \in (0, 1)}\frac{\mathcal{H}^n(B_r(x))}{r^n}>0.
\end{equation}
\end{enumerate}
\end{theorem}
\subsection{Proof of Theorem \ref{thm:main}}
Let us fix a compact $\RCD(K, N)$ space $(X, \dist, \meas)$. 
We recall a result proved in \cite{AmbrosioHondaTewodrose} which states that for all $x \in \mathcal{R}_n$ we have
\begin{equation}\label{eq:shorttime}
\lim_{t \to 0^+}\meas (B_{t^{1/2}}(x))p(x, x, t) =\frac{\omega_n}{(4\pi)^{n/2}}.
\end{equation}

First let us prove the implication from (2) to (1). Assume that  (2) holds. It is trivial from the Bishop-Gromov inequality that (\ref{eq:lower}) holds.
Let $\omega \in D(\Delta_{H, 1})$ with $\Delta_{H, 1}\omega \in D(\delta )$.
Then Theorems \ref{thm:technical} and \ref{thm:divtr} show
\begin{align}\label{al:key}
\int_X\left\langle \frac{c(n)t^{(n+2)/2}g_t-g}{t}, \nabla \omega \right\rangle \di \mathcal{H}^n &=c(n)t^{n/2}\int_X\langle g_t, \nabla \omega \rangle \di \mathcal{H}^n -\frac{1}{t}\int_X \mathrm{tr}\nabla \omega\di \mathcal{H}^n \nonumber \\
&=\frac{c(n)}{4}\int_X\delta (\Delta_{H, 1}\omega) t^{n/2}p(x, x, 2t)\di \mathcal{H}^n+\frac{1}{t}\int_X\delta \omega \di \mathcal{H}^n \nonumber \\
&=\frac{c(n)}{4}\int_X\delta (\Delta_{H, 1}\omega) t^{n/2}p(x, x, 2t)\di \mathcal{H}^n. 
\end{align}
On the other hand (\ref{eq:shorttime}) shows that for any $x \in \mathcal{R}_n$, as $t \to 0^+$
\begin{align}\label{eq:pointwise}
t^{n/2}p(x, x, 2t)&=\frac{1}{\omega_n2^{n/2}}\cdot \frac{\omega_n(2t)^{n/2}}{\mathcal{H}^n(B_{(2t)^{1/2}}(x))} \cdot \mathcal{H}^n(B_{(2t)^{1/2}}(x))p(x, x, 2t) \nonumber \\
&\to \frac{1}{\omega_n2^{n/2}} \cdot 1 \cdot \frac{\omega_n}{(4\pi)^{n/2}}=(8\pi)^{-n/2}.
\end{align}
Since the Bishop-Gromov inequality with (\ref{eq:gaussian}) yields
\begin{equation}
t^{n/2}p(x, x, 2t) \le C(K, n, \mathrm{diam}(X, \dist), \mathcal{H}^n(X)) <\infty,
\end{equation}
letting $t \to 0^+$ in (\ref{al:key}) with the dominated convergence theorem yields that as $t \to 0^+$
\begin{equation}
(\mathrm{RHS})\,\mathrm{of}\,(\ref{al:key}) \to \frac{c(n)}{4(8\pi)^{n/2}}\int_X\delta (\Delta_{H, 1}\omega)\di \mathcal{H}^n=0
\end{equation}
which completes the proof of the desired implication. 

Next we prove the implication from (1) to (2).
Assume that (1) holds. Then for any $\omega \in D(\Delta_{H, 1})$ with $\Delta_{H, 1}\omega \in D(\delta )$ we have
\begin{equation}\label{eq:zeroconv}
\int_X\left\langle c(n)t^{n/2}g_t, \nabla \omega \right\rangle \di \meas-\frac{1}{t}\int_X\mathrm{tr} (\nabla \omega) \di \mathcal{H}^n \to 0.
\end{equation}
Since (\ref{eq:lower}) and (\ref{eq:gaussian}) imply
\begin{equation}
\sup_{t \in (0, 1), x \in X}t^{n/2}p(x, x, 2t)<\infty,
\end{equation} 
the same argument as above yields that
\begin{equation}\label{eq:convergence}
\int_X\left\langle c(n)t^{n/2}g_t, \nabla \omega \right\rangle \di \meas \to \frac{c(n)}{4(8\pi)^{n/2}}\int_X\delta (\Delta_{H, 1}\omega) \di \mathcal{H}^n \in \mathbb{R}.
\end{equation}
In particular combining (\ref{eq:zeroconv}) with (\ref{eq:convergence}) shows that 
\begin{equation}
\frac{1}{t}\int_X\mathrm{tr} (\nabla \omega) \di \mathcal{H}^n
\end{equation}
converges as $t \to 0^+$. This convergence forces  
\begin{equation}
\int_X\mathrm{tr} (\nabla \omega) \di \mathcal{H}^n=0.
\end{equation}
Therefore by (\ref{eq:convergence}) it holds that
\begin{equation}\label{eq:orthogonal}
0=\int_X\delta \Delta_{H, 1}\omega \di \mathcal{H}^n=\int_X\delta (\Delta_{H, 1}\omega) \frac{\di \mathcal{H}^n}{\di \meas}\di \meas.
\end{equation}
For any eigenfunction $f$ of $\Delta$ on $(X, \dist, \meas)$ whose eigenvalue is not zero, letting $\omega=\di f$ in (\ref{eq:orthogonal}) shows
\begin{equation}
\int_Xf\frac{\di \mathcal{H}^n}{\di \meas}\di \meas =0
\end{equation}
which proves that $\frac{\di \mathcal{H}^n}{\di \meas}$ is a constant function because $f$ is an arbitrary eigenfunction. Thus we have (\ref{eq:meashauss}). Then the conclusion follows from Theorem \ref{thm:charnoncollapsed}. \hfill $\square$

\subsection{Weakly asymptotically divergence free}\label{subsec:WADF}
In order to prove Corollary \ref{cor:divfree} let us introduce the following notion:
\begin{definition}[Weakly asymptotically divergence free]\label{def:asy}
 Let $\{T_t\}_{t \in (0, 1)}$ be a family of $L^2$-tensor fields of type $(0, 2)$ on $X$. We say that it is \textit{weakly asymptotically divergence free as $t \to 0^+$} if there exists a dense subset $V$ of $H^{1, 2}_C(T^*(X, \dist, \meas))$ such that for any $\omega \in V$ we have
    \begin{equation}\label{eq:convzero}
       \int_X\langle T_t, \nabla \omega \rangle \di \meas \to 0
    \end{equation}
 as $t \to 0^+$.
\end{definition}
Note that Theorem \ref{thm:main} implies that a family of $L^{\infty}$-tensors (\ref{eq:targetquotient}) is weakly asymptotically divergence free as $t \to 0^+$ if an $\RCD(K, n)$ space $(X, \dist, \meas)$ satisfies $\mathrm{dim}_{\dist, \meas}(X)=n$ because the space 
    \begin{equation}\label{eq:testclass}
       \{\omega \in D(\Delta_{H, 1}); \Delta_{H, 1}\omega \in D(\delta )\}
    \end{equation}
    is dense in $H^{1, 2}_C(T^*(X, \dist, \meas))$, see for instance Remark \ref{rem:expansion}.
    Corollary \ref{cor:divfree} is a direct consequence of Theorem \ref{thm:main} with the following proposition.
\begin{proposition}\label{prop:equivalence}
    Let $\{T_t\}_{t \in (0, 1)}$ be a family of $L^2$-tensor fields of type $(0, 2)$ on $X$ with 
    \begin{equation}
       \limsup_{t \to 0^+}\|T_t\|_{L^2}<\infty
    \end{equation} 
Then the following two conditions are equivalent:
    \begin{enumerate}
       \item $\{T_t\}_{t \in (0, 1)}$ is weakly asymptotically divergence free as $t \to 0^+$.
       \item If $G \in L^2((T^*)^{\otimes 2}(X, \dist, \meas))$ is the $L^2$-weak limit of $T_{t_i}$ for some convergent sequence $t_i \to 0^+$, then $G \in D(\nabla^*)$ with $\nabla^*G=0$. 
    \end{enumerate}
\end{proposition}
 \begin{proof}
   Let us first prove the implication from (1) to (2).
   Assume that $\{T_t\}_{t \in (0, 1)}$ is weakly asymptotically divergence free as $t \to 0^+$.
   Let $V$ be as in Definition \ref{def:asy} and let $t_i, G$ be as in the assumption of (2).
   By definition we have 
     \begin{equation}\label{eq:vanish}
        \int_X\langle G, \nabla \omega \rangle \di \meas=\lim_{i \to \infty}\int_X\langle T_{t_i}, \nabla \omega \rangle \di \meas=0
     \end{equation}
   holds for any $\omega \in V$. Since $V$ is dense in $H^{1, 2}_C(T^*(X, \dist, \meas))$, we have 
     \begin{equation}
        \int_X\langle G, \nabla \omega \rangle \di \meas=0, \quad \forall \omega \in H^{1, 2}_C(T^*(X, \dist, \meas))
     \end{equation}
   which shows $G \in D(\nabla^*)$ with $\nabla^*G=0$.
\\
  Next let us prove the remaining implication. Assume that (2) holds. Let us fix $\omega \in H^{1, 2}_C(T^*(X, \dist, \meas))$.
  If (\ref{eq:convzero}) is not satisfied for this $\omega$, then combining with the $L^2$-weak compactness shows that there exist a convergent sequence $t_i \to 0^+$ and $G \in L^2((T^*)^{\otimes 2}(X, \dist, \meas))$ such that $T_{t_i} \to G$ in the $L^2$-weak topology and 
\begin{equation}
\int_X\langle G, \nabla \omega\rangle \di \meas =\lim_{i \to \infty}\int_X\langle T_{t_i}, \nabla \omega \rangle \di \meas \neq 0
\end{equation}
are satisfied, which contradicts the assumption (2).
\end{proof}

\section{The $L^2$ divergence of the approximate Einstein tensor}\label{sec:stratified}

In this section, we will explain why it is necessary to state the main theorem using the weakly asymptotically divergence free property by giving an example. In fact, we cannot hope that (\ref{eq:targetquotient}) has a limit in a reasonable sense, let alone in $D(\nabla^*)$, more precisely, the $L^2$ convergence of (\ref{eq:targetquotient}) may fail.  To show this we will construct a compact non-collapsed $\RCD(K,3)$ space with $K>1$ such that 
    \begin{equation}
       \left\|\frac{c(3)t^{5/2}g_t-g}{t}\right\|_{L^2}\xrightarrow{t\rightarrow 0^+} +\infty.
    \end{equation}
%We first point out that the computation in section \ref{sec:3} can be generalized to a \textit{smooth} open subset $U$ in an %compact
 %$\RCD (K,N)$ space $(X, \dist, \meas)$ {whose heat kernel $p$ can be expanded by eigenfunctions as in (\ref{eq:expansion1}). 

{
The next proposition is an auxiliary result for our purpose.
%The next proposition indicates that it makes sense, from the point of view of asymptotic behavior of $g_t$, to define the weighted Einstein tensor $G_f^g$, as in (\ref{eq:einsteinweight}) of Definition \ref{def:einstein} on a general %weighted 
%(namely not necessary complete) manifold $(M^n,\dist^g)$%,\vol_f^g)$  
%without boundary. 
Note that an open subset $U \subset X$ %$(U, \dist)$ , \meas \res_U)$
 is said to be \textit{smooth} if for any $y \in U$ there exist an open subset $y \in V \subset U$ and a (not necessary complete) Riemannian manifold $(M^n, g)$ 
such that there exists an isometry  $\Phi:V \to M^n$ %such that %$\Phi_*(\meas \res_U)=\mathrm{vol}^g_f$ and that 
%$\Phi$ is a locally isometry 
 as metric spaces.
%In particular, the definition of $G_f^g$ makes sense for interior points of $\left([0,\infty), \dist^{g_{\mathbb{R}}}, \mathrm{vol}_x^{g_{\mathbb{R}}}\right)$, which we will use in Remark \ref{rem:example}.
}

\begin{proposition}\label{prop:AlmostSmoothBBG} 
    Let $(X,\dist, \mathcal{H}^n)$ be a noncollapsed compact $\RCD(K,n)$ space and let $U\subset X$ be a smooth open subset.
Then %Theorem \ref{thm:bbgweighted} holds on $U$ in the sense that 
     \begin{equation}\label{eq:localasymptotics}
     	 \frac{c(n)t^{n+2/2}g_t-g}{t}\to-\frac{2}{3}G^g
     \end{equation}
    holds uniformly on any compact subset of $U$.
\end{proposition}
\begin{proof}
 Fix  $y\in U$ and take a sufficiently small $\epsilon>0$ such that $B_{\epsilon}(y)\subset U$ and that $\partial B_{\epsilon}(y)$ is smooth. With no loss of generality we can assume {that $(B_{\epsilon}(y), \dist)$ is an open ball in a \textit{closed}  Riemanian manifold $(N^n, h)$.} Let $p_{\epsilon}$ be the Dirichlet heat kernel on $B_{\epsilon}(y)$. Thanks to the smoothness of $\partial B_{\epsilon}(y)$, we know that $p_{\epsilon}$ has the continuous extention, denoted $p_{\epsilon}$ again, to $\overline{B}_{\epsilon}(y) \times \overline{B}_{\epsilon}(y) \times (0, \infty)$ such that $p_{\epsilon}(x, z, t)=0$ whenever $x \in \partial B_{\epsilon}(y)$ which is justified by regularity results for parabolic equations on Euclidean balls.
The key point in the proof of (\ref{eq:localasymptotics}) is to show that the global heat kernel $p$ on $X$ and $p_{\epsilon}$ are exponentially close on $B_{\epsilon}(y)$, that is, for sufficiently small $t$,
    \begin{equation}\label{eq:expclose}
        \sup_{x\in B_{\epsilon}(y)}|p(x,y,t)-p_{\epsilon}(x,y,t)|<C(K,N)e^{-\epsilon^2/6t},	
    \end{equation}
  	where $C(K,n)$ denotes a positive constant with dependence on $K$ and $n$. 
Because after establishing (\ref{eq:expclose}), we can easily complete the proof as follows.

\textbf{Step 1.} The restriction of $p$ to $B_{\epsilon}(y) \times B_{\epsilon}(y) \times (0, \infty)$ is smooth and the expansion (\ref{eq:expansion1}) is satisfied in $C^{\infty}(B_{\epsilon}(y))$ (whenever $\epsilon$ is sufficiently small).

We have several proofs of this fact. One way is to apply the elliptic regularity theorem with elliptic estimates (see for instance \cite{GT}) for the $i$-th eigenfunction $\phi_i$, %in order to show that 
%the $l$-the derivative of $\phi_i$ on $B_{\delta}(x)$ is bounded abobe by $C\lambda_i^\beta$, where $\beta \ge 0$ depends only on $l$ and $C$ depends on the geometry of $B_{\epsilon}(y)$. Thus the $l$-th derivative of the finite sum $\sum_i^me^{-\lambda_it}\phi_i(x)\phi_i(y)$ on $B_{\delta}(y)$ is uniformly bounded because of polynomial grwoth of $\lambda_i$. Thus the Arzel\`a-Ascoli theorem allows us to conclude that
then the expansion (\ref{eq:expansion1}) is satisfied in $C^{l}(B_{\epsilon}(y))$ for any $l\ge 1$, namely we have \textbf{Step 1}.

\textbf{Step 2.} We see that $g_t$ is smooth on $B_{\epsilon}(y)$ and that $g_t(v, v)=(\di_Sp)_{(x,x,2t)}(v,v)$ holds for all $x \in B_{\epsilon}(y)$ and $v \in T_xU$.

This is a direct consequence of (\ref{DefPullBack}) and \textbf{Step 1} (see also the beginning of the proof of Theorem \ref{thm:bbgweighted}, namely (\ref{eq:ds})).

\textbf{Step 3.} It holds that on $B_{\epsilon}(y)$,
\begin{equation}\label{Shi}
 \left|\di_S (p-p_{\epsilon})\right|\le Ce^{-\epsilon^2/7t}.
 \end{equation}

This essentially comes from (\ref{eq:expclose}), we postpone the proof to appendix \ref{appShi}.

\textbf{Final step.} We prove (\ref{eq:localasymptotics}).

The proof of the final step is as follows.
Applying also the previous steps above for $(N^n, \dist^h, \mathrm{vol}^h)$ (instead of $(X, \dist, \mathcal{H}^n)$), denoting by $p^h$ the heat kernel of  $(N^n, \dist^h, \mathrm{vol}^h)$, we have
\begin{equation}
\left|\di_S (p-p^h)\right|\le Ce^{-\epsilon^2/7t}
\end{equation}
on $B_{\epsilon}(y)$. Thus Theorem \ref{thm:bbgweighted} for $(N^n, \dist^h, \mathrm{vol}^h)$ (with the proof) implies that (\ref{eq:localasymptotics}) holds.
%Then since the restriction of $p$ to $B_{\epsilon}(y) \times B_{\epsilon}(y) \times (0, \infty)$ is smooth (see for instance the proof of \cite[Thm.7.20]{Grig}), (\ref{eq:expclose}) implies the power series expansion in $t$ for $p$ and $p_{f, \epsilon}$ are the same. {Therefore this also holds for $p$ and $p_{\phi, h}$.} In particular $p$ has the same expansion as in (\ref{HKexpansion}) on $B_{\epsilon}(y)$ {when we compare $p$ with $p_{\phi, h}$, where the elliptic regularity theorem with elliptic estimates (see for instance \cite{GT}) shows that any eigenfunction on $X$ is smooth on $B_{\epsilon}(y)$ and that the expansion (\ref{eq:expansion1}) is satisfied in $C^{\infty}(B_{\delta}(y))$ for any $0<\delta<\epsilon$.} Then the desired convergence (\ref{eq:localasymptotics}) comes from the same proof of Theorem \ref{thm:bbgweighted}.

Finally we know that it is enough to prove (\ref{eq:expclose}). To this end, applying the Gaussian estimates (\ref{eq:gaussian}), together with the maximum principle yields for small $t>0$
   \begin{equation}\label{eq:ii} 
       \begin{split}	
 		  \sup_{x\in B_{\epsilon}(y)}|p(x,y,t)-p_{\epsilon}(x,y,t)|&\le \sup _{\partial B_{\epsilon}(y)\times (0,t]}(p(x,y,s)-p_{\epsilon}(x,y,s))\\
  		       &\le C_1 e^{C_2 t}\sup_{s\in(0,t]}\frac{e^{-\epsilon^2/5s}}{\meas(B_{\sqrt{s}}(y))}\\
                       &\le C_1 C e^{C_2 t}\sup_{s\in(0,t]}\frac{e^{-\epsilon^2/5s}}{s^{n/2}}\\ 
                       &\le C_1 C e^{C_2 t}\frac{e^{-\epsilon^2/5t}}{t^{n/2}} \\
                       &\le C_1 C e^{C_2 t}e^{-\epsilon^2/6t},
        \end{split}                
   \end{equation}
where we used the Bishop-Gromov inequality for $\mathcal{H}^n$ in the third inequality, and a fact that
   the function $\frac{e^{-\epsilon^2/5s}}{s^{n/2}}$ is monotone increasing for $s\in(0,t]$ when $t$ is small enough.
\end{proof}

\begin{example}\label{excounterexample}
Let $(X, \dist)$ be the spherical suspension of $(\mathbb{S}^2(r), \dist_{\mathbb{S}^2(r)})$ for some $r \in (0, 1)$, where $\mathbb{S}^2(r):=\{x \in \mathbb{R}^3; |x|=r\}$ and $\dist_{\mathbb{S}^2(r)}$ denotes the canonical spherical distance. Note that $(X, \dist, \mathcal{H}^3)$ is a non-collapsed $\RCD (r^{-2}+1, 3)$ space because of \cite[Thm.1.1]{Ketterer}, that $(X \setminus \{p_-, p_+\}, \dist)$ is isometric to a smooth Riemannian manifold $(M^3, g)$, where $p_{\pm}$ denote poles, and that 
\begin{equation}\label{iuuu}
\int_{M^3}|\mathrm{Scal}^g|^2\di \mathcal{H}^3=\infty.
\end{equation}

Let us show the $L^2$ divergence of (\ref{eq:targetquotient}) as $t\to 0^+$ in this example.
Proposition \ref{prop:AlmostSmoothBBG} yields 
      \begin{equation}
              \int_{X}\left\langle \frac{c(3)t^{5/2}g_{t}-g}{t}, T\right\rangle \di \mathcal{H}^3\to-\frac{2}{3}\int_X\left\langle G^{g}, T\right\rangle \di \mathcal{H}^3
      \end{equation}
for any tensor $T$ of type $(0, 2)$ with compact support in $X \setminus \{p_{\pm}\}$. In particular 
\begin{equation}\label{eq:lowerb}
\|G^g\|_{L^2(K)}^2=\left| \int_K\left\langle G^{g}, 1_KG^g\right\rangle \di \mathcal{H}^3\right| \le \frac{3}{2}\liminf_{t\to 0^+}\left\|\frac{c(3)t^{5/2}g_{t}-g}{t}\right\|_{L^2} \cdot \|G^g\|_{L^2(K)}
\end{equation}
for any compact subset $K \subset X \setminus \{p_{\pm}\}$. 
Taking the supremum with respect to $K$ in (\ref{eq:lowerb}), we have
      \begin{equation}\label{eq:blowup}
         \|G^{g}\|_{L^2}\le\frac{3}{2}\liminf_{t\to 0^+}\left\|\frac{c(3)t^{5/2}g_{t}-g}{t}\right\|_{L^2}.
      \end{equation}
Since the left hand side of (\ref{eq:blowup}) is $+\infty$ because of
\begin{equation}
\int_{M^3}|G^g|^2\di \mathcal{H}^3 \ge \frac{1}{3}\int_{M^3}|\langle G^g, g\rangle|^2 \di \mathcal{H}^3 =\frac{1}{12}\int_{M^3}|\mathrm{Scal}^g|^2\di \mathcal{H}^3=\infty,
\end{equation}
the divergence of the right hand side of (\ref{eq:blowup}) follows. 
\end{example}

%{
%\begin{remark}\label{rem:generaldef}
	%Proposition \ref{prop:AlmostSmoothBBG} indicates that it makes sense to define the weighted Einstein tensor $G_f^g$, as in Definition \ref{def:einstein} on a general weighted manifold $(M^n,\dist_g,\vol_f^g)$. In particular, the definition of $G_f^g$ makes sense for interior points of $\left([0,\infty), \dist_{g_{\mathbb{R}}}, \mathrm{vol}_x^{g_{\mathbb{R}}}\right)$, which we will use in the next remark.
%\end{remark}

The compactness of $(X,\dist)$ in Theorem \ref{thm:main} plays a crucial role. We give an example to show that Theorem \ref{thm:main} does not hold without the compactness assumption. For this purpose, we need to define $g_t$ for a general, possibly non-compact, $\RCD(K,N)$ space, see \cite[Definition 3.6]{wNCtoNC} and the discussion therein for the details. %We point out that the Gaussian estimates ensure that the definition there coincides with \eqref{DefPullBack} if the $\RCD$ space is compact.
\begin{example}\label{rem:example} 
Denoting by $g_{\mathbb{R}}$ the canonical Riemannian metric on $\mathbb{R}$, let us consider a smooth metric measure space 
\begin{equation}\label{eq:exam}
\left(\setR, \dist^{g_{\mathbb{R}}}, \mathrm{vol}_x^{g_{\mathbb{R}}}\right), \quad \left( \mathrm{vol}_x^{g_\setR}(A)=\int_Ae^{-x}\di x\right).
\end{equation}
Thanks to (\ref{eq:be}), $(\setR, \dist^{g_{\mathbb{R}}}, \mathrm{vol}_x^{g_{\setR}})$ is an $\RCD(-(N-1)^{-1}, N)$ space for any $N>1$. We compute directly the short time expansion of $g_t$. First, it follows from \cite[Lemma 4.7]{Grig2} that the heat kernel $p$ of $\left(\setR, \dist^{g_{\mathbb{R}}}, \mathrm{vol}_x^{g_{\mathbb{R}}}\right)$ is
\begin{equation}
p(x, y, t)=e^{-\frac t4 +\frac{x+y}2}\frac 1{\sqrt {4\pi t}}e^{-\frac{|x-y|^2}{4t}}.	
\end{equation}
Then, we have
\begin{align}
\di_x p&=\left(\frac12 e^{-\frac t4 +\frac{x+y}2}	\frac 1{\sqrt {4\pi t}}e^{-\frac{|x-y|^2}{4t}}+\frac{x-y}{2t}e^{-\frac t4 +\frac{x+y}2}	\frac 1{\sqrt {4\pi t}}e^{-\frac{|x-y|^2}{4t}}\right)\di x\notag\\
         &=\frac 1{2\sqrt {4\pi t}} e^{-\frac t4 +\frac{x+y}2}	e^{-\frac{|x-y|^2}{4t}}\left(1+\frac{x-y}t\right)\di x.
\end{align}
Finally, keeping in mind $\di x\otimes \di x=g_\setR$, we can compute $g_t$ as follows
\begin{align}
g_t&=\int_\setR \di_x p\otimes \di_x p\ e^{-y}\di y\notag\\
   &=\frac 1{16\pi t}e^{-\frac t2 +x}\int_\setR 	e^{-\frac{|x-y|^2}{2t}}\left(1+\frac{x-y}t\right)^2\di y\ g_{\setR}	\notag\\
   &=\frac 1{16\pi t}e^{-\frac t2+x}\left( \sqrt{2\pi t}+\sqrt{\frac{2\pi}t}\right)g_{\setR}.	
\end{align}
Now the desired expansion reads
\begin{align}
	4\sqrt{8\pi}t^{\frac 32}g_t&=(1+t)e^{-\frac t2 +x}g_{\setR}\notag\\
	               &= e^x(1+t)\left(1-\frac t2+O(t^2)\right)g_{\setR}\notag \\
	               &= e^x g_{\setR}+\frac12 e^x g_{\setR}\ t+O(t^2).
\end{align}
Note that this matches with the general formula obtained for closed manifolds, since in this case $f(x)=x$, we have $\di f\otimes\di f=\di x\otimes \di x=g_\setR$, $\Delta^{g_\setR}f=0$ and $|\nabla f|=1$, recall Definition \ref{def:einstein}. Then from Definition \ref{def:weightedadj} and the fact that for unweighted operator, $\nabla^* g=0$, %which is just integration by parts w.r.t. $1$ dimensional Lebesgue measure, 
we see that
\begin{equation}
\nabla_{x}^*(e^x g_\setR)=	e^x\nabla^*g_\setR-g_\setR(\cdot, e^x\partial_x)+e^x g_\setR(\cdot, \partial_x) =-e^x\di x+e^x\di x=0.
\end{equation}

This computation shows for the $\RCD(-(N-1)^{-1},N)$ space $\left(\setR, \dist^{g_\setR}, \mathrm{vol}_x^{g_\setR}\right)$, the second principal term of $g_t$ is divergence free, nevertheless it carries a non-constant density $e^{-x}$.

%the space (\ref{eq:exam}) is an $\RCD(-(N-1)^{-1}, N)$ space because of the global-to-local property \cite[Prop. 7.7]{AmbrosioMondinoSavare2}.
 %Note that this space has finite mass, namely, \ $\mathrm{vol}_x^{g_{\mathbb{R}}}([0,\infty))=1$. Thus it follows from \cite[Prop.6.5 and 6.7]{GigliMondinoSavare13} (or \cite[Prop.7.5]{AmbrosioHonda})  that the inclusion $H^{1,2}\left([0,\infty), \dist^{g_{\mathbb{R}}}, \mathrm{vol}_x^{g_{\mathbb{R}}}\right)\hookrightarrow L^2([0,\infty),\mathrm{vol}_x^{g_{\mathbb{R}}})$ is a compact operator. 
%It is trivial that $G^g_x$ is divergence free because of $\Delta^{g_\setR} x=0$ (see the proof of Proposition \ref{prop:directproof}).
%On the other hand 
%applying Proposition \ref{prop:AlmostSmoothBBG} to $U=(0,\infty)$, %along with proposition \ref{prop:directproof}, 
%we see that 
%\begin{equation}\label{eq:localasymptotics}
     	 %\frac{c(1)t^{1+2/2}g_t-e^xg}{t}\to-\frac{2}{3}G^g_x
     %\end{equation}
    %holds uniformly on any compact subset of $U$.
%asymptotically divergence free (because of (\ref{eq:localasymptotics})). %Then it is easy to see that the weighted Einstein tensor is divergence free (see the proof of Proposition \ref{prop:directproof} in particular (\ref{eq:tobeconstant})). 
%However the weight function $x$ is not a constant, namely $\left([0,\infty), \dist^{g_{\mathbb{R}}}, \mathrm{vol}_x^{g_{\mathbb{R}}}\right)$ is not non-collapsed. This observation allows us to conclude that in general Theorem \ref{thm:main} fails if the space is non-compact.
%the same conclusion as in Theorem \ref{thm:main} cannot be expected in the noncompact setting.
\end{example}
%}

\section{Appendices}
\subsection{Spectral analysis on compact $\RCD$ spaces}\label{sec:app}
In this appendix we provide a Rellich type compactness for $1$-forms, Theorem \ref{thm:rellich}, which in particular proves that the space (\ref{eq:testclass}) is dense in $H^{1 2}_C(T^*(X, \dist, \meas))$;
\begin{equation}\label{eq:dense}
\overline{\{\omega \in D(\Delta_{H, 1}); \Delta_{H, 1}\omega \in D(\delta)\}}=H^{1, 2}_C(T^*(X, \dist, \meas))
\end{equation}
Let us mention that $h_{H, t}\omega$ is in (\ref{eq:testclass}) for any $\omega \in L^2(T^*(X, \dist, \meas))$ and any $t>0$, which gives another proof of (\ref{eq:dense}) without the compactness of $(X, \dist)$, where $h_{H, t}$ is the heat flow acting on $L^2(T^*(X, \dist, \meas))$ associated with the energy;
\begin{equation}
\omega \mapsto \frac{1}{2}\int_X(|\di \omega|^2+|\delta \omega|^2)\di \meas,
\end{equation}
as discussed in \cite[(3.6.18)]{Gigli}.
The authors believe that the Rellich type compactness result has an independent interest from the point of view of the spectral analysis on compact $\RCD(K, N)$ spaces, see also \cite{Honda01}.

For the proof, we need several analytic notions, including the local Sobolev spaces $H^{1, p}(U, \dist, \meas)$, the domain of local Laplacian $D(\Delta, U) (\subset H^{1, 2}(U, \dist, \meas))$ with the Laplacian $\Delta_U=\Delta$ for any open subset $U$ of $X$ and so on. We refer \cite{AmbrosioHonda, AmbrosioHonda2, HKST} for the detail.
Let us emphasize that the $\RCD(K, N)$ condition for a metric measure space $(X, \dist, \meas)$ plays an essential role to establish:
\begin{enumerate}
\item{(Good cut-off function, \cite[Lem.3.1]{MondinoNaber})} for any $x \in X$ and all $0<r<R<\infty$, there exists $\phi \in D(\Delta) \cap \mathrm{Lip}_b(X, \dist)$ such that $0\le \phi \le 1$ holds, that $\phi \equiv 1$ holds on $B_r(x)$, that $\supp \phi \subset B_R(x)$ holds, and that $|\nabla \phi| +|\Delta \phi| \le C(K, N, r, R)$ holds for $\meas$-a.e. $x \in X$;
\item{(Hessian estimates for harmonic functions)} For any harmonic function $f$ on $B_R(x) \subset X$ with $|\nabla f| \le L$, that is, $f \in D(\Delta, B_R(x))$ with $\Delta f \equiv 0$, and for any $r<R$, we have
\begin{equation}\label{eq:hessest}
\int_{B_r(x)}|\mathrm{Hess}_f|^2\di \meas \le C(K, N, r, R, L).
\end{equation}
\end{enumerate}
Note that the Hessian of a harmonic function $f$ as above is well-defined as a measurable tensor over $B_R(x)$ because of the locality of the Hessian proved in \cite[Prop.3.3.24]{Gigli}, see also \cite[(1.1)]{BPS}.
The proof of (\ref{eq:hessest}) is easily done by applying (\ref{eq:bochner}) with the good cut-off function constructed in (1).

Finally let us recall a useful notation from the convergence theory; 
\begin{equation}
\Psi (\epsilon_1, \epsilon_2, \ldots, \epsilon_l;c_1, c_2, \ldots, c_m)
\end{equation}
denotes a function $\Psi: (\mathbb{R}_{>0})^l \times \mathbb{R}^m \to (0, \infty)$ satisfying
\begin{equation}
\lim_{(\epsilon_1, \ldots, \epsilon_k) \to 0} \Psi  (\epsilon_1, \epsilon_2, \ldots, \epsilon_l;c_1, c_2, \ldots, c_m)=0, \quad \forall c_i.
\end{equation}
The authors know that the following result is independently obtained in \cite{B} as an application of the heat flow when the paper is finalized.  Our proof is based on $\delta$-splitting maps which is different from that of \cite{B}. 
\begin{theorem}[Rellich compactness]\label{thm:rellich}
Let $(X, \dist, \meas)$ be a compact $\RCD(K, N)$ space. Then the canonical inclusion map:
\begin{equation}\label{eq:inc}
H^{1, 2}_C(T^*(X, \dist, \meas)) \hookrightarrow L^2(T^*(X, \dist, \meas))
\end{equation}
is a compact operator.
\end{theorem}
\begin{proof}
With no loss of generality we can assume that $\meas (X)=1$ and $N>1$.
Let $\omega_i$ be a bounded sequence in $H^{1, 2}_C(T^*(X, \dist, \meas))$.
By the $L^2$-weak compactness with no loss of generality we can assume that $\omega_i$ $L^2$-weakly converge to some $\omega \in L^2(T^*(X, \dist, \meas))$. Our goal is to prove that this is an $L^2$-strong convergence.

Let us remark that thanks to \cite[Prop.3.4.6]{Gigli} (recall that for any $\omega \in L^2(T^*(X, \dist, \meas))$, $\omega \in W^{1, 2}_C(T^*(X, \dist, \meas)$ holds if and only if $\omega^{\sharp} \in W^{1, 2}_C(T(X, \dist, \meas))$ holds), we have $|\omega_i|^2 \in H^{1, 1}(X, \dist, \meas)$ with $|\nabla |\omega_i|^2| \le 2|\nabla \omega_i| |\omega_i|$ for $\meas$-a.e. $x \in X$.
In particular the Sobolev embedding theorem proved in \cite[Thm.5.1]{HK} yields
\begin{equation}\label{eq:lpbd}
\sup_i\||\omega_i|^2\|_{L^{p_N}}<\infty,
\end{equation}
where $p_N:=N/(N-1)$ because a Poincar\'e inequality is satisfied \cite[Thm.1]{Rajala}, and the Bishop-Gromov inequality implies the inequality $\meas (B_s(y)) \ge C(s/r)^N\meas (B_r(x))$ for all $x \in X$, $y \in B_r(x)$ and $s \in (0, r]$, see \cite[(21)]{HK}.

Fix $\epsilon>0$ and put $n:=\mathrm{dim}_{\dist, \meas}(X)$.  For any $x \in \mathcal{R}_n$ there exists $r_x>0$ such that for any $r \in (0, r_x)$ there exists a harmonic map $\Phi_{r, x}=(\phi_{r, x, 1}, \phi_{r, x, 2},\ldots, \phi_{r, x, n}):B_{2r}(x) \to \mathbb{R}^n$ (that is, each $\phi_{r, x, i}$ is a harmonic function on $B_{2r}(x)$) such that $|\nabla \phi_{r, x, i}| \le C(K, N)$ holds for any $i$, that 
\begin{equation}\label{eq:l2est}
\frac{1}{\meas (B_{2r}(x))}\int_{B_{2r}(x)}\left|\langle\nabla \phi_{r, x, i}, \nabla \phi_{r, x, j}\rangle-\delta_{ij}\right|\di \meas + \frac{r^2}{\meas (B_{2r}(x))}\int_{B_{2r}(x)}|\mathrm{Hess}_{\phi_{r, x, i}}|^2\di \meas \le \epsilon
\end{equation}
holds for all $i, j$ (see \cite[Prop.1.4]{BPS}). Note that the $L^2$-weak convergence of $\omega_i$ to $\omega$ yields that $\langle \di \phi_{r, x, j}, \omega_i\rangle$ $L^2$-weakly converge to $\langle \di \phi_{r, x, j}, \omega\rangle$ on $B_{2r}(x)$ for any $j$.

On the other hand applying \cite[Prop.3.4.6]{Gigli} (with a good cut-off function as above) again yields $\langle \di \phi_{r, x, j}, \omega_i\rangle \in H^{1, 1}(B_r(x), \dist, \meas)$ with 
\begin{equation}
|\nabla \langle \di \phi_{r, x, j}, \omega\rangle |\le |\mathrm{Hess}_{\phi_{r, x, j}}||\omega_i| + |\nabla \phi_{r, x, j}| |\nabla \omega_i|, \quad \mathrm{for}\,\meas-\mathrm{a.e.}\,x\in B_r(x).
\end{equation} 
{For the reader's convenience}, let us provide a proof of the above. Take $\phi \in D(\Delta) \cap \mathrm{Lip}_b(X, \dist)$ such that $0 \le \phi \le 1$ holds, that $\phi \equiv 1$ holds on $B_r(x)$, that $\supp \phi \subset B_{2r}(x)$ and that $|\nabla \phi|+|\Delta f| \le C(K, N, r)$ holds for $\meas$-a.e. $x \in X$. Then since $\phi \phi_{r, x, j} \in D(\Delta) \cap \mathrm{Lip}_b(X, \dist)$, applying \cite[Prop.3.4.6]{Gigli}  yields 
$\langle \di (\phi\phi_{r, x, j}), \omega_i\rangle \in H^{1, 1}(X, \dist, \meas)$ with 
\begin{equation*}
|\nabla \langle \di (\phi \phi_{r, x, j}), \omega\rangle |\le |\mathrm{Hess}_{\phi \phi_{r, x, j}}||\omega_i| + |\nabla \phi_{r, x, j}| |\nabla (\phi \omega_i)|, \quad \mathrm{for}\,\meas-\mathrm{a.e.}\,x\in X.
\end{equation*} 
Restricting this observation to $B_r(x)$ with the locality properties of the gradient (for instance \cite[Thm.2.2.6]{Gigli}) and of the Hessian \cite[Prop.3.3.24]{Gigli} proves the desired statement.

In particular  (\ref{eq:hessest}) shows
\begin{equation}
\sup_i\|\langle \di \phi_{r, x, j}, \omega_i\rangle \|_{H^{1, 1}(B_r(x), \dist, \meas)}<\infty.
\end{equation}
Therefore applying the Rellich compactness theorem for $H^{1, 1}$-functions proved in \cite[Thm.8.1]{HK} shows that $\langle \di \phi_{r, x, j}, \omega_i\rangle$ $L^p$-strongly converge to $\langle \di \phi_{r, x, j}, \omega\rangle$ on $B_r(x)$ for all $p \in [1, p_N)$. By (\ref{eq:lpbd}) we see that $\langle \di \phi_{r, x, j}, \omega_i\rangle$ $L^2$-strongly converge to $\langle \di \phi_{r, x, j}, \omega\rangle$ on $B_r(x)$ for any $j$.

Let 
\begin{equation}
A(r, x):=\left\{y \in B_r(x);|\langle \nabla \phi_{r, x, i}, \nabla \phi_{r, x, j}\rangle(y)-\delta_{ij}| \le \epsilon^{1/2},\forall i,\,\forall j\right\}.
\end{equation}
Then the Markov inequality with (\ref{eq:l2est}) shows
\begin{equation}\label{eq:markov}
\frac{\meas (B_r(x)\setminus A(r, x))}{\meas(B_r(x))}\le \epsilon^{1/2}.
\end{equation}
Note that for any $\eta \in L^2(T^*(X, \dist, \meas))$ 
\begin{equation}
\left| |\eta|^2(y) -\sum_j\langle \di \phi_{r, x, j}, \eta \rangle^2(y)\right|\le \Psi \left(\epsilon ; n\right)|\eta|^2, \quad \mathrm{for}\,\mathrm{a.e.}\,y\in A(r, x).
\end{equation}
See also \cite[(5.36) and (5.37)]{AHPT}.
Applying the Vitali covering theorem to a family $\mathcal{F}:=\{\overline{B}_r(x)\}_{x \in \mathcal{R}_n, r<r_x}$ yields that there exists a pairwise disjoint subfamily $\{\overline{B}_{r_j}(x_j)\}_{j \in \mathbb{N}}$ of $\mathcal{F}$ such that 
\begin{equation}
\mathcal{R}_n\setminus \bigsqcup_{j=1}^k\overline{B}_{r_j}(x_j) \subset \bigcup_{j \ge k+1}\overline{B}_{5r_j}(x_j), \quad \forall k
\end{equation}
holds. Take $k_0$ with $\sum_{j\ge k_0+1}\meas (B_{r_j}(x_j))<\epsilon$. Then by (\ref{eq:markov}) we have
\begin{align}
\meas \left(X \setminus \bigsqcup_{j=1}^{k_0}A(r_j, x_j)\right) &\le \meas \left(X \setminus \bigsqcup_{j=1}^{k_0}B_{r_j} (x_j)\right) + \sum_{j=1}^{k_0}\meas (B_{r_j}(x_j) \setminus A(r_j, x_j)) \nonumber \\
&\le \sum_{j\ge k_0+1}\meas (B_{5r_j}(x_j)) + \epsilon^{1/2}\sum_{j=1}^{k_0}\meas (B_{r_j}(x_j)) \nonumber \\
&\le C(K, N)\sum_{j \ge k_0+1}\meas (B_{r_j}(x_j)) +\epsilon^{1/2}\nonumber \\
&\le \Psi(\epsilon; K, N).
\end{align}
Thus for any sufficiently large $i$ we have
\begin{align}\label{al:vitali}
&\int_{X}|\omega_i|^2\di \meas \nonumber \\
&=\sum_{j=1}^{k_0}\int_{A(r_j, x_j)}|\omega_i|^2 \di \meas +\int_{X \setminus \bigsqcup_{j=1}^{k_0}A(r_j, x_j)}|\omega_i|^2\di \meas \nonumber \\
&\le \sum_{j=1}^{k_0}\sum_{l=1}^n\int_{A(r_j, x_j)}(\langle \di \phi_{r_j, x_j, l}, \omega_i\rangle^2 +\Psi (\epsilon; n)|\omega_i|^2) \di \meas+\meas \left(X\setminus \bigsqcup_{j=1}^{k_0}A(r_j, x_j)\right)^{1/q_N}\||\omega_i|^2\|_{L^{p_N}} \nonumber \\
&\le \sum_{j=1}^{k_0}\sum_{l=1}^n\int_{A(r_j, x_j)}\langle \di \phi_{r_j, x_j, l}, \omega \rangle^2 \di \meas+\Psi (\epsilon; n) \sup_m\|\omega_m\|_{L^2}^2+\Psi(\epsilon; K, N)\sup_m\||\omega_m|^2\|_{L^{p_N}} \nonumber \\
&\le  \sum_{j=1}^{k_0}\sum_{l=1}^n\int_{A(r_j, x_j)}(1+\Psi(\epsilon; n))|\omega|^2 \di \meas +\Psi(\epsilon; K, N)( \sup_m\|\omega_m\|_{L^2}^2+ \sup_m\||\omega_m|^2\|_{L^{p_N}}) \nonumber \\
&\le \int_X|\omega|^2 \di \meas +\Psi(\epsilon; K, N)( \sup_m\|\omega_m\|_{L^2}^2+ \sup_m\||\omega_m|^2\|_{L^{p_N}}),
\end{align}
where $q_N$ is the conjugate exponent of $p_N$. Since $\epsilon$ is arbitrary, (\ref{al:vitali}) shows that 
\begin{equation}
\limsup_{i \to \infty}\int_X|\omega_i|^2\di \meas \le \int_X|\omega|^2\di \meas
\end{equation}
which completes the proof of the $L^2$-strong convergence of $\omega_i$ to $\omega$.
\end{proof}

The following corollary is a direct consequence of Corollary \ref{cor:equiv} and Theorem \ref{thm:rellich}  (see for instance the appendix of \cite{Honda0}).
\begin{corollary}\label{cor:spec}
The spectrum of the Hodge Laplacian $\Delta_{H, 1}$ acting on $1$-forms is discrete and unbounded. If we denote the spectrum by
\begin{equation}
0\le \lambda_{(H, 1), 1} \le \lambda_{(H, 1), 2} \le \lambda_{(H, 1), 3} \le \cdots \le \lambda_{(H, 1), k} \le \cdots \to \infty
\end{equation}
counted with multiplicities, then corresponding eigen-$1$-forms $\omega_1, \omega_2, \ldots$ with $\|\omega_k\|_{L^2}=1$ give an orthogonal basis of $L^2(T^*(X, \dist, \meas))$.
\end{corollary}
\begin{remark}\label{rem:expansion}
Under the same notation as in Corollary \ref{cor:spec}, it is easy to see that for any $\omega \in H^{1, 2}_H(T^*(X, \dist, \meas))$,
\begin{equation}\label{eq:h12expansion}
\omega=\sum_i\left(\int_X\langle \omega, \omega_i\rangle \di \meas\right) \omega_i
\end{equation}
in $H^{1, 2}_H(T^*(X, \dist, \meas))$. In particular (\ref{eq:h12expansion}) also holds in $H^{1, 2}_C(T^*(X, \dist, \meas))$ because of (\ref{eq:boch}).
\end{remark}
\begin{remark}
As an immediate consequence of Theorem \ref{thm:rellich}, we are able to prove a similar spectral decomposition result as in Corollary \ref{cor:spec} for the \textit{connection Laplacian} $\Delta_{C, 1}$ acting on $1$-forms. Moreover the technique provided in the proof of Theorem \ref{thm:rellich} allows us to prove similar decomposition results for the connection Laplacians acting on differential forms and tensor fields of any type. Compare with \cite{Honda01, Honda02}.  
\end{remark}

\subsection{Proof of (\ref{Shi})}\label{appShi}

   In order to complete the proof of Proposition \ref{prop:AlmostSmoothBBG}, we recall the following local derivative estimates which are well-known. %(see for instance \cite{localheat}). %Although the conclusion will be not directly applied to getting (\ref{Shi}), it is useful to understand the strategy of the proof. 
%The proof is the same as Shi's estimates for derivatives of curvature tensor in \cite{SHI}.  For reader's convenience we provide a complete proof following closely to \cite[section 1.4]{CZ}. 
For our purpose, it is enough to consider the case when the total space is complete because we recall that $B_{\epsilon}(y)$ appeared in the proof of Proposition \ref{prop:AlmostSmoothBBG} is actually an open subset of a closed Riemannian manifold $(N^n, h)$.
   \begin{lemma}\label{lem:loc_deri_esti}
   Let $(U^n, g)$ be an $n$-dimensional complete Riemannian manifold with $\mathrm{Ric}^g\ge -K g$ for some $K>0$, and let $u(x,t)$ be a smooth solution to the heat equation on $B_{2r}(p)\times (0,T]$, for some $p \in U^n$, $0<r \le 1$ and $T>0$. Then 
   \begin{equation}\label{eq:loc_deri_esti}
   |\nabla u|^2(x, t)\le C_n\|u\|_{L^\infty(B_{2r}(p)\times (0, T])}^2\left(\frac{1}{r^{2}}+\frac{1}{t}+K\right)	\quad \forall x\in B_r(p),\text{ $\forall t\in (0,T]$.}
   \end{equation}
   \end{lemma}
\begin{proof}%[proof of lemma \ref{lem:loc_deri_esti}]

This is a direct consequence of a result of Souplet-Zhang in \cite[Theorem 1.1]{SZ} which states that if $u$ is a positive solution of the heat equation on $B_{2r}(p)\times (0, T)$ with $u \le L$, then
\begin{equation}\label{eq:SZ}
\frac{|\nabla u|^2}{u^2}\le C_n\left(\frac{1}{r^2}+\frac{1}{t}+K\right) \left(1+\log \frac{L}{u}\right),\quad \text{on $B_r(p)\times [T/2, T)$}.
\end{equation}
Because in our setting, letting $M=\|u\|_{L^\infty(B_{2r}(p)\times (0,T])}$ and, without loss of generality, we can assume that $M>0$. Consider a positive solution $u+2M$ of the heat equation on $B_{2r}(p)\times (0,T]$. For any $t \in (0, T)$, finding $T_0 \in (0, T)$ with $t \in (\frac{T_0}{2}, T_0)$ and then applying (\ref{eq:SZ}) for this solution on $B_{2r}(p)\times (\frac{T_0}{2}, T_0)$ show (\ref{eq:loc_deri_esti}) because of $M \le u+2M \le 3M$.
\end{proof}

Let us return to the proof of (\ref{Shi}). For any $V\in T_y U$, consider a smooth function $u(x,t)=g(\nabla_y (p-p_{f,\epsilon}), V)$. Observe that $u(x,t)$ satisfies the heat equation since
\begin{equation}
	\frac{\partial}{\partial t}u=g(\nabla_y \Delta^g_x(p-p_{f,\epsilon}), V)=\Delta^g_xg(\nabla_y (p-p_{f,\epsilon}), V)=\Delta^g u,
\end{equation}
where we used %Bochner's formula to move $\Delta^g$ out of the inner product, since 
a fact that $V$ is independent of $x$. We then apply Lemma \ref{lem:loc_deri_esti} twice to derive that for fixed $y$, take any $x\in B_{\epsilon/4}(y)$ and then take $t$ small enough, we have 
 \begin{equation}
 	|\nabla u|(x, t)\le \frac{C}{\sqrt{t}}\|\nabla_y p-\nabla_y p_{f,\epsilon}\|_{L^{\infty}(B_{\epsilon/2}(y))}|V| \le \frac Ct \|p-p_{f,\epsilon}\|_{L^{\infty}(B_{\epsilon}(y))}|V|\le \frac Ct e^{-\epsilon^2/6t}|V|.
 \end{equation}
 Then considering the case when $x=y$, for any $W\in T_y U$, %apply Lemma \ref{lem:loc_deri_esti} to $y$ 
we get 
 \begin{align}
 | \di_S (p-p_{f,\epsilon})(W,V)|\le |W||\nabla u|(y, t)%\notag\\
 %&\le \frac{C}{\sqrt{t}}\|\nabla_y p-\nabla_y p_{f,\epsilon}\|_{L^{\infty}(B_{\epsilon/2}(y))}|V||W|\notag\\
 %&\le \frac Ct \|p-p_{f,\epsilon}\|_{L^{\infty}(B_{\epsilon}(y))}|V||W|\notag\\
 \le \frac Ct e^{-\epsilon^2/6t}|V||W|.
 \end{align}

  Since $V,W$ is arbitrary, we have %it follows that for $x,y$ close enough and $t$ small enough,
 \begin{equation}
 \left|\di_S (p-p_{f,\epsilon})\right|\le \frac Cte^{Ct}e^{-\epsilon^2/6t}	\le Ce^{-\epsilon^2/7t}	
 \end{equation}
which completes the proof of (\ref{Shi}).

  %Altogether, for small $t$, it holds
    %\begin{equation}
	   %\sup_{x\in B_R(y)}|p(x,y,t)-p_R(x,y,t)|\le C(K,N,n)e^{C_2 t}t^{-n/2}e^{-R^2/5t}\le C(K,N,n)e^{-R^2/6t}.
    %\end{equation}
  %In the last inequality we used that $t^{-n/2}e^{-C/t}$ is bounded for small positive $t$. 
    
    %There is a smallest open heat kernel $p_U$ on $U$, such that when the global heat kernel $p$ restricted on $U$, $p_U\le p$. This $p_U$ is defined in \cite[Theorem 3.6]{HeatKernelOpenMfd} by 
    %\begin{equation}
    	%p_U=\sup\left\{p_{\Omega}: \Omega\subseteq U \text{ is relative compact with smooth boundary}\right\}
    %\end{equation}
    %where $p_\Omega$ is the Dirichlet heat kernel on $\Omega$. In particular, when $\Omega=B_R(y)$ for some $y\in U$ and $R>0$. We see by definition that  
   %\begin{equation}
   	   %p_R\le  p_U\le p
   %\end{equation}
%and from (\ref{eq:expclose}) that
    %\begin{equation}
    	 %p-p_R\le C(k,n)e^{-R^2/6t}.
    %\end{equation}
%As a consequence, all $3$ heat kernels above satisfies (\ref{eq:expclose}), thus they all have the same local expansion. 
 %By \ref{thm:WeightedExpansion}, $p_U$ has the desired expansion and convergence properties, so it holds for $p$ as well. Once we know that $p$ has the desired expansion, together with the fact that $U$ is relatively compact in $X$, Theorem \ref{thm:bbgweighted} remains true on $U$.

\end{document}